\documentclass{amsart}
\usepackage{amsmath, amssymb, amsthm, mathrsfs, amscd}
\usepackage{graphicx}

\theoremstyle{plain} 
 \newtheorem{thm}{Theorem}[section]
 \newtheorem{lem}[thm]{Lemma}
 \newtheorem{cor}[thm]{Corollary}
 \newtheorem{prop}[thm]{Proposition}
 \newtheorem{claim}[thm]{Claim}

\theoremstyle{definition}
  \newtheorem{defn}[thm]{Definition}

\theoremstyle{remark}
  \newtheorem{rem}[thm]{Remark}

\renewcommand{\mod}{{\rm Mod}}
\renewcommand{\pmod}{{\rm PMod}}

\newcommand{\cal}{\mathcal}
\newcommand{\ci}[2]{\cite[#1]{#2}}

\newcommand{\cali}{\mathcal{I}}

\newcommand{\calk}{\mathcal{K}}
\newcommand{\calc}{\mathcal{C}}
\newcommand{\calt}{\mathcal{T}}

\newcommand{\lk}{{\rm Lk}}

\begin{document}

\title[The Johnson kernel and the Torelli group]{The co-Hopfian property of the Johnson kernel and the Torelli group}
\author{Yoshikata Kida}
\address{Department of Mathematics, Kyoto University, 606-8502 Kyoto, Japan}
\email{kida@math.kyoto-u.ac.jp}
\date{November 16, 2010}

\begin{abstract}
For all but finitely many compact orientable surfaces, we show that any superinjective map from the complex of separating curves into itself is induced by an element of the extended mapping class group. 
We apply this result to proving that any finite index subgroup of the Johnson kernel is co-Hopfian. 
Analogous properties are shown for the Torelli complex and the Torelli group. 
\end{abstract}

\maketitle


\section{Introduction}\label{sec-intro}

Let $S=S_{g, p}$ be a connected, compact and orientable surface of genus $g$ with $p$ boundary components.
Unless otherwise stated, we assume that a surface satisfies these conditions.
The {\it extended mapping class group} $\mod^*(S)$ for $S$ is defined as the group of isotopy classes of homeomorphisms from $S$ onto itself, where isotopy may move points in the boundary of $S$.
A simple closed curve in $S$ is said to be {\it essential} in $S$ if it is neither homotopic to a single point of $S$ nor isotopic to a boundary component of $S$.
The {\it complex of curves} for $S$, denoted by $\calc(S)$, is defined as the abstract simplicial complex whose vertices are isotopy classes of essential simple closed curves in $S$ and simplices are non-empty finite sets of such isotopy classes having mutually disjoint representatives.
This complex was introduced by Harvey \cite{harvey}.
The group $\mod^*(S)$ naturally acts on $\calc(S)$ as simplicial automorphisms. 
It is known that any simplicial automorphism of $\calc(S)$ is generally induced by an element of $\mod^*(S)$, as proved in \cite{iva-aut}, \cite{kork-aut} and \cite{luo}.
This fact is used to describe any isomorphism between finite index subgroups of $\mod^*(S)$.

A {\it superinjective map} $\phi \colon \calc(S)\rightarrow \calc(S)$, introduced by Irmak \cite{irmak1}, is defined as a simplicial map $\phi \colon \calc(S)\rightarrow \calc(S)$ preserving non-adjacency of two vertices of $\calc(S)$.
Any superinjective map from $\calc(S)$ into itself is easily seen to be injective.
In \cite{be-m}, \cite{bm-ar}, \cite{irmak1}, \cite{irmak2} and \cite{irmak-ns}, any superinjective map from $\calc(S)$ into itself is shown to be surjective and thus induced by an element of $\mod^*(S)$.
This leads to the co-Hopfian property of any finite index subgroup of $\mod^*(S)$, where a group $\Gamma$ is said to be {\it co-Hopfian} if any injective homomorphism from $\Gamma$ into itself is surjective.

Several variants of the complex of curves are introduced to follow the same line as above for some important subgroups of $\mod^*(S)$.
An essential simple closed curve in $S$ is said to be {\it separating} in $S$ if its complement in $S$ is not connected.
We define the Johnson kernel $\calk(S)$ for $S$ as the subgroup of $\mod^*(S)$ generated by all Dehn twists about separating curves in $S$.
Note that $\calk(S)$ is a normal subgroup of $\mod^*(S)$.
The {\it complex of separating curves} for $S$, denoted by $\calc_s(S)$, is defined to be the full subcomplex of $\calc(S)$ spanned by all vertices of $\calc(S)$ corresponding to separating curves in $S$.
It is shown in \cite{bm}, \cite{bm-add} and \cite{kida} that for all but finitely many surfaces $S$, any simplicial automorphism of $\calc_s(S)$ is induced by an element of $\mod^*(S)$, as precisely stated in Theorem \ref{thm-cs-aut}. 
This result is applied to proving that the abstract commensurator of $\calk(S)$ is naturally isomorphic to $\mod^*(S)$. 
The aim of this paper is to prove that any superinjective map from $\calc_s(S)$ into itself is surjective and is thus induced by an element of $\mod^*(S)$.
As a result, any finite index subgroup of $\calk(S)$ is shown to be co-Hopfian.

\begin{thm}\label{thm-cs}
Let $S=S_{g, p}$ be a surface satisfying one of the following three conditions: $g=1$ and $p\geq 3$; $g=2$ and $p\geq 2$; or $g\geq 3$ and $p\geq 0$. 
Then
\begin{enumerate}
\item any superinjective map from $\calc_s(S)$ into itself is induced by an element of $\mod^*(S)$.
\item if $\Gamma$ is a finite index subgroup of $\calk(S)$ and if $f\colon \Gamma \rightarrow \calk(S)$ is an injective homomorphism, then there exists a unique $\gamma_0\in \mod^*(S)$ satisfying the equality $f(\gamma)= \gamma_0\gamma \gamma_0^{-1}$ for any $\gamma \in \Gamma$. 
In particular, $\Gamma$ is co-Hopfian.
\end{enumerate}
\end{thm}

Most of the paper is devoted to the proof of assertion (i).
We omit the proof of assertion (ii) since the process to derive it from assertion (i) is already discussed in Section 5 of \cite{bm} and Section 6.3 of \cite{kida}.
We obtain similar conclusions for the Torelli complex $\calt(S)$ and the Torelli group $\cali(S)$ for $S$, which are defined in Section \ref{sec-pre}.

\begin{thm}\label{thm-t}
Let $S$ be the surface in Theorem \ref{thm-cs}. 
Then
\begin{enumerate}
\item any superinjective map from $\calt(S)$ into itself is induced by an element of $\mod^*(S)$.
\item if $\Lambda$ is a finite index subgroup of $\cali(S)$ and if $h\colon \Lambda \rightarrow \cali(S)$ is an injective homomorphism, then there exists a unique $\lambda_0\in \mod^*(S)$ satisfying the equality $h(\lambda)= \lambda_0\lambda \lambda_0^{-1}$ for any $\lambda \in \Lambda$. 
In particular, $\Lambda$ is co-Hopfian.
\end{enumerate}
\end{thm}

The proof of this theorem uses Theorem \ref{thm-cs} and is presented in Section \ref{sec-tor}.
We refer to Remark 1.3 in \cite{kida} for known facts on the complex of separating curves and the Torelli complex for a surface which is not dealt with in Theorems \ref{thm-cs} and \ref{thm-t}.
Among other things, it is notable that $\calc_s(S_{2, 1})$ consists of countably infinitely many $\aleph_0$-regular trees.
This is a direct consequence of Theorem 7.1 in \cite{kls}.

Although the same conclusions as Theorems \ref{thm-cs} and \ref{thm-t} for closed surfaces are asserted in Theorems 1.6 and 1.8 of Brendle-Margalit's paper \cite{bm}, their argument contains a gap as precisely discussed in Remark \ref{rem-gap}.
The present paper fills this gap by considering not only closed surfaces but also non-closed ones, while Brendle-Margalit deal with only closed ones.
In fact, assertion (i) in Theorem \ref{thm-cs} are proved by induction on $g$ and $p$, whose first step is the case $(g, p)=(1, 3)$.

This paper is organized as follows. 
In Section \ref{sec-pre}, we introduce the terminology and notation employed throughout the paper and review the definition of the complexes and subgroups of the mapping class group discussed above. 
In Section \ref{sec-d}, we introduce the simplicial graph $\cal{D}$ associated with $S_{1, 2}$ and provide basic properties of it, which will be used in subsequent sections. 
In Section \ref{sec-g1}, we obtain the conclusion of Theorem \ref{thm-cs} for surfaces of genus one.
In Section \ref{sec-const}, given a surface $S$ with its genus at least two and a superinjective map $\phi \colon \calc_s(S)\rightarrow \calc_s(S)$, we explain how to extend $\phi$ to a simplicial map $\Phi \colon \calc(S)\rightarrow \calc(S)$. 
Using the map $\Phi$, we prove surjectivity of $\phi$ for $S_{2, 2}$ in Section \ref{sec-22} and prove it for the remainder of surfaces other than $S_{3, 0}$ by induction on $g$ and $p$ in Section \ref{sec-g2}. 
We deal with $S_{3, 0}$ independently in Section \ref{sec-30}.
Finally, we deduce Theorem \ref{thm-t} from Theorem \ref{thm-cs} in Section \ref{sec-tor}.


\section{Preliminaries}\label{sec-pre}

\subsection{Terminology}\label{subsec-term}

Let $S=S_{g, p}$ be a surface of genus $g$ with $p$ boundary components. 
We define $V(S)$ to be the set of isotopy classes of essential simple closed curves in $S$.
When there is no confusion, we mean by a curve either an essential simple closed curve in $S$ or the isotopy class of it. 
An essential simple closed curve $a$ in $S$ is said to be {\it separating} in $S$ if $S\setminus a$ is not connected, and otherwise $a$ is said to be {\it non-separating} in $S$.
Whether an essential simple closed curve in $S$ is separating in $S$ or not depends only on its isotopy class. 
A pair of non-separating curves in $S$, $\{ a, b \}$, is called a {\it bounding pair (BP)} in $S$ if $a$ and $b$ are disjoint and non-isotopic and if $S\setminus (a \cup b)$ is not connected. 
These conditions depend only on the isotopy classes of $a$ and $b$.

We mean by a {\it handle} a surface homeomorphic to $S_{1, 1}$ and mean by a {\it pair of pants} a surface homeomorphic to $S_{0, 3}$.
Let $a$ be a separating curve in $S$.
If $a$ cuts off a handle from $S$, then $a$ is called an {\it h-curve} in $S$.
If $a$ cuts off a pair of pants from $S$, then $a$ is called a {\it p-curve} in $S$.

Suppose that $\partial S$ is non-empty.
A simple arc $l$ in $S$ is said to be {\it essential} in $S$ if
\begin{itemize}
\item $\partial l$ consists of two distinct points of $\partial S$;
\item $l$ meets $\partial S$ only at its end points; and
\item $l$ is not isotopic relative to $\partial l$ to an arc in $\partial S$. 
\end{itemize}
Let $A(S)$ denote the set of isotopy classes of essential simple arcs in $S$, where isotopy may move the end points of arcs, keeping them staying in $\partial S$.
We say that two elements of $V(S)\sqcup A(S)$ are {\it disjoint} if they have disjoint representatives.
Frequently, we do not distinguish an element of $A(S)$ and its representative if there is no confusion.
An essential simple arc $l$ in $S$ is said to be {\it separating} in $S$ if $S\setminus l$ is not connected.
Otherwise $l$ is said to be {\it non-separating} in $S$.
Whether an essential simple arc in $S$ is separating in $S$ or not depends only on its isotopy class. 
Given two components $\partial_1$, $\partial_2$ of $\partial S$, we say that an essential simple arc $l$ in $S$ {\it connects $\partial_1$ and $\partial_2$} if one of the end points of $l$ lies in $\partial_1$ and another in $\partial_2$.


\subsection{The mapping class group and its subgroups}\label{subsec-mcg}

Let $S$ be a surface.
The {\it mapping class group} $\mod(S)$ for $S$ is defined as the subgroup of $\mod^*(S)$ consisting of all isotopy classes of orientation-preserving homeomorphisms from $S$ onto itself.
The {\it pure mapping class group} $\pmod(S)$ for $S$ is defined as the subgroup of $\mod^*(S)$ consisting of all isotopy classes of orientation-preserving homeomorphisms from $S$ onto itself that fix each boundary component of $S$ as a set.
Both $\mod(S)$ and $\pmod(S)$ are normal subgroups of $\mod^*(S)$ of finite index.

For each $a\in V(S)$, we denote by $t_{a}\in \pmod(S)$ the {\it (left) Dehn twist} about $a$. 
The Johnson kernel $\calk(S)$ for $S$ is the subgroup of $\pmod(S)$ generated by all Dehn twists about separating curves in $S$. 
The {\it Torelli group} $\cali(S)$ for $S$ is defined as the subgroup of $\pmod(S)$ generated by all Dehn twists about separating curves in $S$ and all elements of the form $t_{a}t_{b}^{-1}$ with $\{ a, b\}$ a BP in $S$.
Note that $\calk(S)$ and $\cali(S)$ are normal subgroups of $\mod^*(S)$.
Originally, the Torelli group are defined in a different way when the number of boundary components of $S$ is at most one. 
Thanks to \cite{johnson} and \cite{powell}, the Torelli group defined originally is equal to the one defined above.


\subsection{Simplicial complexes associated to a surface}\label{subsec-comp}

Let $S$ be a surface.
We denote by $i\colon V(S)\times V(S)\rightarrow \mathbb{Z}_{\geq 0}$ the {\it geometric intersection number}, i.e., the minimal cardinality of the intersection of representatives for two elements of $V(S)$. 
Let $\Sigma(S)$ denote the set of non-empty finite subsets $\sigma$ of $V(S)$ with $i(\alpha, \beta)=0$ for any $\alpha, \beta \in \sigma$.
We extend $i$ to the symmetric function on $(V(S)\sqcup \Sigma(S))^2$ so that $i(\alpha, \sigma)=\sum_{\beta \in \sigma}i(\alpha, \beta)$ and $i(\sigma, \tau)=\sum_{\beta \in \sigma, \gamma \in \tau}i(\beta, \gamma)$ for any $\alpha \in V(S)$ and $\sigma, \tau \in \Sigma(S)$.
We say that two elements $\sigma$, $\tau$ of $V(S)\sqcup \Sigma(S)$ are {\it disjoint} if $i(\sigma, \tau)=0$, and otherwise we say that they {\it intersect}. 

For each $\sigma \in \Sigma(S)$, we denote by $S_{\sigma}$ the surface obtained by cutting $S$ along all curves in $\sigma$. 
When $\sigma$ consists of a single curve $a$, we denote it by $S_a$ for simplicity. 
We often identify a component of $S_{\sigma}$ with a complementary component of a tubular neighborhood of a one-dimensional submanifold representing $\sigma$ in $S$ if there is no confusion.
If $Q$ is a component of $S_{\sigma}$, then $V(Q)$ is naturally identified with a subset of $V(S)$.

The complex of curves $\calc(S)$ for $S$ is the abstract simplicial complex such that the set of vertices and simplices are $V(S)$ and $\Sigma(S)$, respectively.
Let $V_s(S)$ denote the subset of $V(S)$ consisting of separating curves in $S$.
The complex of separating curves for $S$, denoted by $\calc_s(S)$, is defined as the full subcomplex of $\calc(S)$ spanned by $V_s(S)$.

Let $V_{bp}(S)$ denote the set of isotopy classes of BPs in $S$. 
We often regard an element of $V_{bp}(S)$ as an edge of $\calc(S)$. 
The {\it Torelli complex} for $S$, denoted by $\calt(S)$, is defined to be the abstract simplicial complex such that the set of vertices is the disjoint union $V_{s}(S)\sqcup V_{bp}(S)$, and a non-empty finite subset $\sigma$ of $V_{s}(S)\sqcup V_{bp}(S)$ is a simplex of $\calt(S)$ if and only if any two elements of $\sigma$ are disjoint.
The Torelli complex (with additional structure and for closed surfaces) were introduced by Farb-Ivanov \cite{farb-ivanov}.

Connectivity of $\calc_s(S)$ and $\calt(S)$ is already discussed in \cite{farb-ivanov} and \cite{mv} when $S$ is closed.
Applying Putman's idea in Lemma 2.1 of \cite{putman-conn} to prove connectivity of a simplicial complex on which $\pmod(S)$ acts, we obtain the following lemma without effort.

\begin{lem}
Let $S=S_{g, p}$ be a surface and assume one of the following three conditions: $g=1$ and $p\geq 3$; $g=2$ and $p\geq 2$; and $g\geq 3$ and $p\geq 0$.
Then both $\calc_s(S)$ and $\calt(S)$ are connected.
\end{lem}

The proof of this lemma uses a family of simple closed curves in $S$, described in Figure \ref{fig-mcg} (a), such that the Dehn twists about them generate $\pmod(S)$.
The same kind of argument to apply Putman's idea appears in the proof of Lemmas \ref{lem-d-conn}, \ref{lem-p-conn} and \ref{lem-hp-conn}.

\subsection{Superinjective maps}

Let $S$ be a surface, and let $X$ be one of the simplicial complexes $\calc(S)$, $\calc_s(S)$ and $\calt(S)$. 
We denote by $V(X)$ the set of vertices of $X$. 
Note that a map $\phi \colon V(X)\rightarrow V(X)$ defines a simplicial map from $X$ into itself if and only if $i(\phi(a), \phi(b))=0$ for any two vertices $a, b\in V(X)$ with $i(a, b)=0$. 
We mean by a {\it superinjective map} $\phi \colon X\rightarrow X$ a simplicial map $\phi \colon X\rightarrow X$ satisfying $i(\phi(a), \phi(b))\neq 0$ for any two vertices $a, b\in V(X)$ with $i(a, b)\neq 0$.
This property was introduced by Irmak \cite{irmak1} when $X=\calc(S)$. 

Any superinjective map $\phi \colon X\rightarrow X$ is injective.
For if there were two distinct vertices $a, b\in V(X)$ with $\phi(a)=\phi(b)$, then superinjectivity of $\phi$ would imply $i(a, b)=0$.
Since $a$ and $b$ are distinct, we can choose $c\in V(X)$ with $i(a, c)=0$ and $i(b, c)\neq 0$.
By superinjectivity of $\phi$, we have $i(\phi(a), \phi(c))=0$ and $i(\phi(b), \phi(c))\neq 0$.
This contradicts the equality $\phi(a)=\phi(b)$.

\subsection{Known results}

To prove surjectivity of a superinjective map $\phi \colon \calc_s(S)\rightarrow \calc_s(S)$ when $\calc_s(S)$ is connected, it is enough to show that $\phi$ sends the link of each vertex $\alpha$ of $\calc_s(S)$ onto the link of $\phi(\alpha)$.
We apply induction on $g$ and $p$ to proving it because the link of a vertex of $\calc_s(S)$ consists of the complexes of separating curves for surfaces with $g$ or $p$ smaller than those of $S$.
The following theorems will be used to complete this inductive argument.

\begin{thm}[\cite{iva-aut}, \cite{kork-aut}, \cite{luo}]\label{thm-cc}
Let $S=S_{g, p}$ be a surface with $3g+p-4>0$. 
If $(g, p)\neq (1, 2)$, then any automorphism of $\calc(S)$ is induced by an element of $\mod^*(S)$. 
If $(g, p)=(1, 2)$, then any automorphism of $\calc(S)$ that preserves vertices corresponding to separating curves in $S$ is induced by an element of $\mod^*(S)$.
\end{thm}

Any superinjective map from $\calc(S)$ into itself is shown to be surjective in \cite{be-m}, \cite{bm-ar}, \cite{irmak1}, \cite{irmak2} and \cite{irmak-ns}. 
More generally, the following theorem is obtained.

\begin{thm}[\cite{sha}]\label{thm-sha}
Let $S=S_{g, p}$ be a surface with $3g+p-4>0$. 
Then any injective simplicial map from $\calc(S)$ into itself is surjective.
\end{thm}

The same conclusion as Theorem \ref{thm-cc} is obtained for the complexes of separating curves for certain surfaces.

\begin{thm}[\cite{bm}, \cite{kida}]\label{thm-cs-aut}
Let $S=S_{g, p}$ be a surface satisfying one of the following three conditions: $g=1$ and $p\geq 3$; $g=2$ and $p\geq 2$; or $g\geq 3$ and $p\geq 0$. 
Then any automorphism of $\calc_s(S)$ is induced by an element of $\mod^*(S)$.
\end{thm}


\section{Graph $\cal{D}$}\label{sec-d}

Throughout this section, we put $R=S_{1, 2}$ and focus on the simplicial graph $\cal{D}=\cal{D}(R)$ defined as follows.

\medskip

\noindent {\bf Graph $\cal{D}=\cal{D}(R)$.} The set of vertices of $\cal{D}$ is defined to be $V_s(R)$ and denoted by $V(\cal{D})$. 
Two vertices $\alpha, \beta \in V(\cal{D})$ are connected by an edge of $\cal{D}$ if and only if we have $i(\alpha, \beta)=4$.

\medskip

The aim of this section is to prove the following:

\begin{prop}\label{prop-d}
Any injective simplicial map from $\cal{D}$ into itself is surjective.
\end{prop}

We fix the notation employed throughout this section.
Let $\partial_1$ and $\partial_2$ denote the two boundary components of $R$.
We note that there is a one-to-one correspondence between the isotopy classes of separating curves in $R$ and essential simple arcs in $R$ connecting $\partial_1$ and $\partial_2$, where isotopy of essential simple arcs in $R$ may move the end points of arcs, keeping them staying in $\partial R$. 
Namely, one associates to a separating curve $\alpha$ in $R$ an arc connecting $\partial_1$ and $\partial_2$ and disjoint from $\alpha$, which is uniquely determined up to isotopy.
This arc is denoted by $l_{\alpha}$ (see Figure \ref{fig-nsarc} (a)).
Conversely, for each essential simple arc $l$ in $R$ connecting $\partial_1$ and $\partial_2$, the separating curve in $R$ corresponding to $l$ is obtained as a boundary component of a regular neighborhood of the union $l\cup \partial R$ in $R$.

\begin{figure}
\includegraphics[width=11cm]{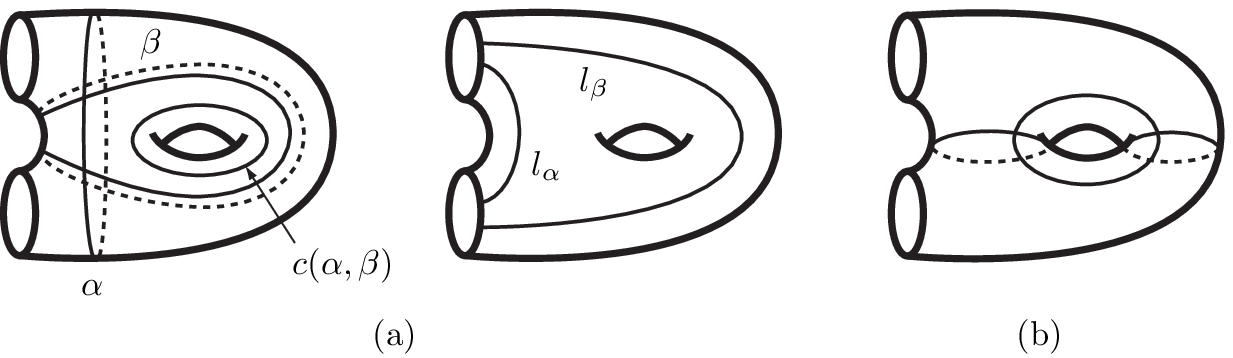}
\caption{}\label{fig-nsarc}
\end{figure}

Note that if for each $k=1, 2$, $l_1^k$ and $l_2^k$ are essential simple arcs in $R$ such that
\begin{itemize}
\item each of $l_1^k$ and $l_2^k$ connects $\partial_1$ and $\partial_2$; and
\item $l_1^k$ and $l_2^k$ are disjoint and non-isotopic,
\end{itemize}
then there exists a homeomorphism $F$ from $R$ onto itself preserving an orientation of $R$ and satisfying $F(\partial_1)=\partial_1$, $F(\partial_2)=\partial_2$ and $F(l_j^1)=l_j^2$ for each $j=1, 2$.
For if we cut $R$ along $l_1^k$ and $l_2^k$, then we obtain an annulus $A_k$.
One can then construct a homeomorphism from $A_1$ onto $A_2$ sending arcs in $\partial A_1$ corresponding to $l_j^1$ to arcs in $\partial A_2$ corresponding to $l_j^2$ for each $j=1, 2$ and inducing a desired homeomorphism from $R$ onto itself.

\begin{lem}\label{lem-arc-disjoint}
For any two distinct vertices $\alpha, \beta \in V(\cal{D})$, we have $i(\alpha, \beta)=4$ if and only if $l_{\alpha}$ and $l_{\beta}$ are disjoint.
\end{lem}

\begin{proof}
Using the criterion on intersection numbers in Expos\'e 3, Proposition 10 of \cite{flp}, one can check that the curves $\alpha$ and $\beta$ described in Figure \ref{fig-nsarc} (a) satisfy $i(\alpha, \beta)=4$.
The ``if" part thus follows from the argument right before the lemma.

Pick two vertices $\alpha$, $\beta$ of $\cal{D}$ with $i(\alpha, \beta)=4$.
Let $A$ and $B$ be representatives of $\alpha$ and $\beta$, respectively, with $|A\cap B|=4$.
We denote by $H$ the handle cut off by $A$ from $R$ and naturally identify $H$ with a subset of $R$.
The intersection $B\cap H$ consists of two simple arcs in $H$, denoted by $b_1$ and $b_2$.
Neither $b_1$ nor $b_2$ are isotopic relative to their end points to an arc in $\partial H$ because $A$ and $B$ intersect minimally.
It follows that $b_1$ and $b_2$ are essential simple arcs in $H$.
The arcs $b_1$ and $b_2$ are isotopic because otherwise $\beta$ would be non-separating in $R$.

We denote by $P$ the pair of pants cut off by $A$ from $R$ and naturally identify it with a subset of $R$.
The intersection $B\cap P$ consists of two essential simple arcs in $P$, which are isotopic.
Let $b_3$ and $b_4$ denote the two components of $B\cap P$.

Fix an orientation of $A$.
For each $j=1, 2$, we put $\partial b_j=\{ p_j, q_j \}$ so that $p_1$, $q_1$, $q_2$ and $p_2$ appear along $A$ in this order.
For each $k=3, 4$, the arc $b_k$ connects neither $p_1$ and $q_1$ nor $p_2$ and $q_2$ because otherwise $b_k$ and either $b_1$ or $b_2$ would form a simple closed curve.
For each $k=3, 4$, the arc $b_k$ connects neither $p_1$ and $q_2$ nor $p_2$ and $q_1$ because $b_k$ is separating in $P$.
It turns out that $b_3$ and $b_4$ connect either $p_1$ and $p_2$ or $q_1$ and $q_2$.

Let $I$ and $J$ denote the components of $A\setminus \{ p_1, p_2\}$ and $A\setminus \{ q_1, q_2\}$, respectively, that contain no point of $A\cap B$.
Note that $I$ and $J$ lie in the same component of $H\setminus B$.
We may assume that $I$ and $\partial_1$ (resp.\ $J$ and $\partial_2$) lie in the same component of $P\setminus B$.
Pick essential simple arcs $r_1$ and $r_2$ in $P$ such that
\begin{itemize}
\item $r_1$ connects a point of $\partial_1$ with a point of $I$, and $r_2$ connects a point of $\partial_2$ with a point of $J$; and
\item both $r_1$ and $r_2$ are disjoint from $B\cap P$.
\end{itemize}
Since $I$ and $J$ lie in the same component of $H\setminus B$, we can find an essential simple arc $r_3$ in $H$ disjoint from $B\cap H$ and connecting the point of $r_1\cap I$ with the point of $r_2\cap J$.
We define $r$ as the union $r_1\cup r_2\cup r_3$, which is an essential simple arc in $R$ connecting $\partial_1$ and $\partial_2$ and disjoint from $B$.
Pick an essential simple arc $l$ in $P$ connecting $\partial_1$ and $\partial_2$ and disjoint from $r_1$ and $r_2$.
Since $l$ is an essential simple arc in $R$ disjoint from $A$ and $r$, the ``only if" part of the lemma follows.
\end{proof}

The last lemma and the observation right before the lemma imply that for any two edges $\{ \alpha_1, \beta_1\}$, $\{ \alpha_2, \beta_2\}$ of $\cal{D}$, there exists an element $f$ of $\pmod(R)$ with $f(\alpha_1)=\alpha_2$ and $f(\beta_1)=\beta_2$.
For any edge $\{ \alpha, \beta \}$ of $\cal{D}$, we can find a non-separating curve in $R$ disjoint from $\alpha$ and $\beta$, which is uniquely determined up to isotopy, because the surface obtained by cutting $R$ along $l_{\alpha}$ and $l_{\beta}$ is an annulus.
This non-separating curve is denoted by $c(\alpha, \beta)\in V(R)$ (see Figure \ref{fig-nsarc} (a)).

\subsection{Geometric properties of $\cal{D}$}\label{subsec-fiber}

The following basic property of $\cal{D}$ is shown by applying Putman's idea in Lemma 2.1 of \cite{putman-conn}.

\begin{lem}\label{lem-d-conn}
The graph $\cal{D}$ is connected.
\end{lem}

\begin{proof}
Let $\alpha$ be the curve in Figure \ref{fig-nsarc} (a).
We pick a vertex $\gamma \in V(\cal{D})$ and show that $\alpha$ and $\gamma$ can be connected by a path in $\cal{D}$.
We define $T$ as the set consisting of the Dehn twists about the curves in Figure \ref{fig-nsarc} (b) and their inverses.
It is known that $\pmod(R)$ is generated by $T$ (see \cite{gervais}).
Since $\alpha$ and $\gamma$ are sent to each other by an element of $\pmod(R)$, we can find elements $h_1,\ldots, h_n$ of $T$ with $\gamma =h_1\cdots h_n\alpha$.
We note that for each $h\in T$, either $h\alpha =\alpha$ or $h\alpha$ and $\alpha$ are connected by an edge of $\cal{D}$.
The sequence of vertices of $\cal{D}$,
\[\alpha,\ h_1\alpha,\ h_1h_2\alpha,\ \ldots,\ h_1\cdots h_n\alpha =\gamma,\]
therefore forms a path in $\cal{D}$.  
\end{proof}

We make observation on a fibered structure in the link of each vertex of $\cal{D}$.
To describe it, we recall simplicial graphs associated to $S_{1, 1}$ and to $S_{0, 4}$.

\medskip

\noindent {\bf Graph $\cal{F}(X)$.} Let $X$ be a surface homeomorphic to $S_{1, 1}$ or $S_{0, 4}$.
We define $\cal{F}(X)$ as the simplicial graph such that the set of vertices of $\cal{F}(X)$ is $V(X)$ and two vertices $\alpha, \beta \in V(X)$ are connected by an edge of $\cal{F}(X)$ if and only if we have $i(\alpha, \beta)=1$ when $X$ is homeomorphic to $S_{1, 1}$, and we have $i(\alpha, \beta)=2$ when $X$ is homeomorphic to $S_{0, 4}$.

\medskip

It is known that $\cal{F}(X)$ is isomorphic to the Farey graph (see Section 3.2 in \cite{luo}).
We mean by a {\it triangle} of a simplicial graph $\cal{G}$ a subgraph of $\cal{G}$ consisting of three vertices and three edges.
Let us say that two triangles $\Delta$, $\Delta'$ in a simplicial graph $\cal{G}$ are {\it chain-connected} in $\cal{G}$ if there exists a sequence of triangles of $\cal{G}$, $\Delta_1,\ldots, \Delta_n$, with $\Delta_1=\Delta$ and $\Delta_n=\Delta'$ and with $\Delta_j\cap \Delta_{j+1}$ an edge of $\cal{G}$ for each $j=1,\ldots, n-1$.
The following properties of the Farey graph $\cal{F}$ are notable:
\begin{itemize}
\item Any vertex of $\cal{F}$ is contained in a triangle of $\cal{F}$.
\item Any two triangles of $\cal{F}$ are chain-connected in $\cal{F}$.
\item For any edge $e$ of $\cal{F}$, there exist exactly two triangles of $\cal{F}$ containing $e$.
\end{itemize}
Using these facts, one can show that any injective simplicial map from $\cal{F}$ into itself is surjective.

In the rest of this subsection, we fix a vertex $\alpha \in V(\cal{D})$.
We define $L$ to be the link of $\alpha$ in $\cal{D}$ and define $V(L)$ to be the set of vertices of $L$.
We denote by $H$ the handle cut off by $\alpha$ from $R$ and denote by $\cal{F}$ the graph $\cal{F}(H)$ defined above.

Let $\pi \colon L\rightarrow \cal{F}$ be the simplicial map defined by $\pi(\beta)=c(\alpha, \beta)$ for each $\beta \in V(L)$.
Simpliciality of $\pi$ is proved as follows.
If $\{ \beta, \gamma \}$ is an edge of $L$, then one can find essential simple arcs $l_{\alpha}$, $l_{\beta}$ and $l_{\gamma}$ in $R$ such that
\begin{itemize}
\item for each $\delta \in \{ \alpha, \beta, \gamma \}$, $l_{\delta}$ connects $\partial_1$ and $\partial_2$ and is disjoint from a representative of $\delta$; and 
\item $l_{\alpha}$, $l_{\beta}$ and $l_{\gamma}$ are pairwise disjoint.
\end{itemize}
Let $Q$ denote the surface obtained by cutting $R$ along $l_{\alpha}$, which is a handle.
Note that $\pi(\beta)$ (resp.\ $\pi(\gamma)$) is the only curve in $Q$ disjoint from $l_{\beta}$ (resp.\ $l_{\gamma}$).
Since $l_{\beta}$ and $l_{\gamma}$ are disjoint, we obtain either $\pi(\beta)=\pi(\gamma)$ or $i(\pi(\beta), \pi(\gamma))=1$.

Let $h\in \mod(R)$ be the half twist about $\alpha$ exchanging $\partial_1$ and $\partial_2$ and being the identity on $H$, which satisfies $h^2=t_{\alpha}$. 
We now describe the fiber of $\pi$ over a triangle of $\cal{F}$.

\begin{lem}\label{lem-d-line}
Pick two curves $b$, $c$ in $H$ with $i(b, c)=1$. 
We set
\[B=\{ \, \beta \in V(L)\mid \pi(\beta)=b\, \},\quad \Gamma =\{ \, \gamma \in V(L)\mid \pi(\gamma)=c\, \}.\]
Then we have a numbering of elements, $B=\{ \beta_n\}_{n\in \mathbb{Z}}$ and $\Gamma =\{ \gamma_m\}_{m\in \mathbb{Z}}$, such that
\begin{itemize}
\item $h(\beta_n)=\beta_{n+1}$ and $h(\gamma_m)=\gamma_{m+1}$ for any $n, m\in \mathbb{Z}$; and
\item the full subgraph of $\cal{D}$ spanned by $B\cup \Gamma$ is the bi-infinite line with $\beta_n$ adjacent to $\gamma_n$ and $\gamma_{n+1}$ for each $n\in \mathbb{Z}$.
\end{itemize}
\end{lem}

\begin{proof}
We describe the curves $b$ and $c$ as in Figure \ref{fig-dline} (a) and define $\beta_0$ as the curve in $R$ described in Figure \ref{fig-dline} (b).
Note that $\beta_0$ belongs to $B$.
We say that two vertices $u$, $v$ of a simplicial graph $\cal{G}$ {\it lie in a diagonal position of two adjacent triangles of $\cal{G}$} if there exist two triangles $\Delta_1$, $\Delta_2$ of $\cal{G}$ such that $u\in \Delta_1$, $v\in \Delta_2$ and $\Delta_1\cap \Delta_2$ is an edge of $\cal{G}$ containing neither $u$ nor $v$.
One can check that the two vertices $\alpha$, $\beta_0$ of $\cal{F}(R_b)$ lie in a diagonal position of two adjacent triangles of $\cal{F}(R_b)$.
It follows that for each vertex $\beta$ of $B$, $\alpha$ and $\beta$ lie in a diagonal position of two adjacent triangles of $\cal{F}(R_b)$ because any two edges of $\cal{D}$ are sent to each other by an element of $\pmod(R)$.
Since the cyclic group generated by $h$ acts transitively on the set of triangles of $\cal{F}(R_b)$ containing $\alpha$, it also acts transitively on the set of vertices of $B$.
We thus have the equality $B=\{ h^n(\beta_0)\}_{n\in \mathbb{Z}}$.

\begin{figure}
\includegraphics[width=12cm]{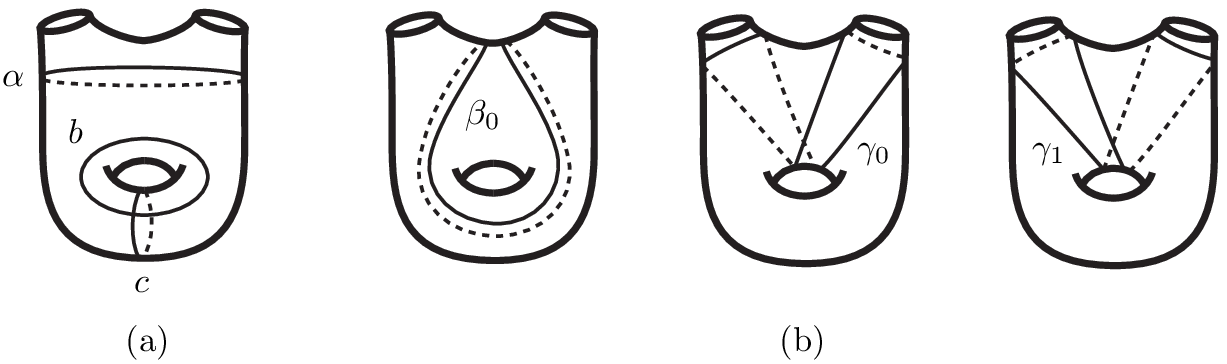}
\caption{}\label{fig-dline}
\end{figure}

Let $\gamma_0$ and $\gamma_1=h(\gamma_0)$ be the curves in $R$ described in Figure \ref{fig-dline} (b).
Note that $\gamma_0$ and $\gamma_1$ belong to $\Gamma$.
The argument in the previous paragraph implies the equality $\Gamma =\{ h^n(\gamma_0)\}_{n\in \mathbb{Z}}$.
We put $\beta_n=h^n(\beta_0)$ and $\gamma_m=h^m(\gamma_0)$ for any $n, m\in \mathbb{Z}$. 

Using the criterion on intersection numbers in Expos\'e 3, Proposition 10 of \cite{flp}, one can check the equality $i(\beta_n, \beta_m)=8\left| n-m\right|$ for any $n, m\in \mathbb{Z}$.
It follows that any two distinct elements of $B$ are not adjacent in $\cal{D}$. 
The same property holds for elements of $\Gamma$ in place of those of $B$.
For each $n\in \mathbb{Z}$, we obtain the equality $i(\beta_0, \gamma_n)=4\left|2n-1\right|$ by using the same criterion in \cite{flp}.
It follows that $\gamma_0$ and $\gamma_1$ are exactly the elements of $\Gamma$ adjacent to $\beta_0$ in $\cal{D}$. 
Applying $h$, we see that the full subgraph of $\cal{D}$ spanned by $B\cup \Gamma$ is the bi-infinite line with $\beta_n$ adjacent to $\gamma_n$ and $\gamma_{n+1}$ for each $n\in \mathbb{Z}$.
\end{proof}

Lemma \ref{lem-d-line} shows that for any edge $\{ b, c\}$ of $\cal{F}$ and any vertex $\beta$ in $\pi^{-1}(b)$, there exists a vertex $\gamma$ in $\pi^{-1}(c)$ with $\{ \beta, \gamma \}$ an edge of $L$.
Connectivity of $\cal{F}$ and the fiber of $\pi$ over any edge of $\cal{F}$ therefore implies connectivity of $L$.

\begin{figure}
\includegraphics[width=12cm]{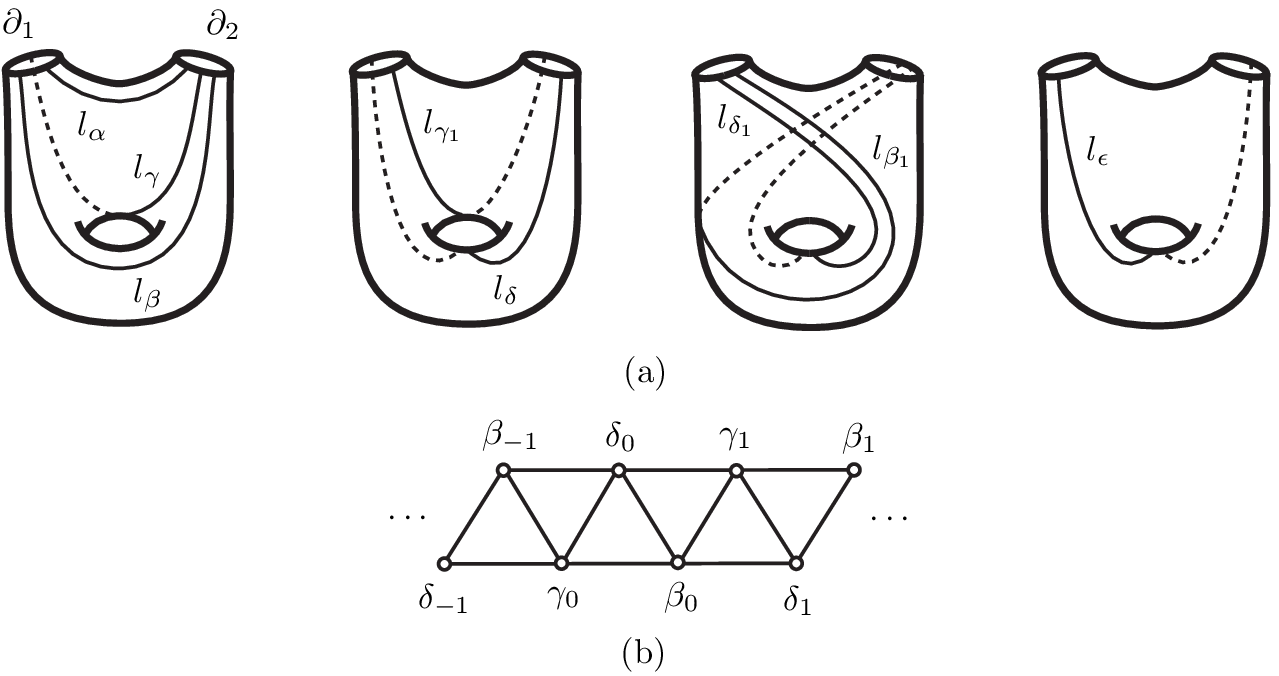}
\caption{}\label{fig-fiber}
\end{figure}

Choose three vertices $\beta$, $\gamma$ and $\delta$ of $\cal{D}$ so that the three arcs $l_{\beta}$, $l_{\gamma}$ and $l_{\delta}$ are described as in Figure \ref{fig-fiber} (a).
Note that each of $l_{\beta}$, $l_{\gamma}$ and $l_{\delta}$ is disjoint from $l_{\alpha}$.
Setting $\beta_n=h^n(\beta)$, $\gamma_n=h^n(\gamma)$ and $\delta_n=h^n(\delta)$ for each $n\in \mathbb{Z}$, we obtain the equalities
\[\pi^{-1}(\pi(\beta))=\{ \beta_n\}_{n\in \mathbb{Z}},\quad \pi^{-1}(\pi(\gamma))=\{ \gamma_n\}_{n\in \mathbb{Z}},\quad \pi^{-1}(\pi(\delta))=\{ \delta_n\}_{n\in \mathbb{Z}}\]
by Lemma \ref{lem-d-line}. 
The fiber of the map $\pi \colon L\rightarrow \cal{F}$ over the triangle of $\cal{F}$ consisting of the three vertices $\pi(\beta)$, $\pi(\gamma)$ and $\pi(\delta)$ is the sequence of triangles described in Figure \ref{fig-fiber} (b).


\subsection{Proof of Proposition \ref{prop-d}}

Let $\psi \colon \cal{D}\rightarrow \cal{D}$ be an injective simplicial map.
For each $\alpha \in V(\cal{D})$, we denote by $L_{\alpha}$ the link of $\alpha$ in $\cal{D}$.
To prove surjectivity of $\psi$, it is enough to show that for each $\alpha \in V(\cal{D})$, the map $\psi_{\alpha}\colon L_{\alpha}\rightarrow L_{\psi(\alpha)}$ defined as the restriction of $\psi$ is surjective since $\cal{D}$ is connected as proved in Lemma \ref{lem-d-conn}.

In what follows, we fix $\alpha \in V(\cal{D})$ and put $L=L_{\alpha}$. 
We denote by $V(L)$ the set of vertices of $L$. 
To prove surjectivity of $\psi_{\alpha}$, we show the following two lemmas.

\begin{lem}\label{lem-tri-three}
For each edge $e$ of $L$, there exist exactly three triangles of $L$ containing $e$.
\end{lem}

\begin{lem}\label{lem-seq-tri}
Any two triangles of $L$ are chain-connected in $L$.
\end{lem}

Using Lemmas \ref{lem-tri-three} and \ref{lem-seq-tri}, we can show surjectivity of $\psi_{\alpha}$ as follows.
Lemma \ref{lem-tri-three} and injectivity of $\psi_{\alpha}$ imply that if $\Delta$ is a triangle of $L$, then $\psi_{\alpha}(L)$ contains any triangle containing an edge of the triangle $\psi_{\alpha}(\Delta)$.
By Lemma \ref{lem-seq-tri}, the image $\psi_{\alpha}(L)$ contains any triangle of $L$.
Surjectivity of $\psi_{\alpha}$ follows because any vertex of $L$ is contained in a triangle of $L$.
In the rest of this subsection, we prove Lemmas \ref{lem-tri-three} and \ref{lem-seq-tri}.
Let $H$ denote the handle cut off by $\alpha$ from $R$, and let $\cal{F}$ denote the graph $\cal{F}(H)$ introduced in Section \ref{subsec-fiber}.

\begin{proof}[Proof of Lemma \ref{lem-tri-three}]
We note that any two edges of $L$ are sent to each other by an element of the stabilizer of $\alpha$ in $\mod(R)$.
This fact follows from Lemma \ref{lem-d-line} and transitivity of the action of $\mod(H)$ on the set of edges of $\cal{F}$.
Let $\{ \beta, \gamma \}$ be an edge of $L$.
We define separating curves $\delta$ and $\epsilon$ in $R$ so that the arcs $l_{\delta}$ and $l_{\epsilon}$ are described in Figure \ref{fig-fiber} (a), respectively.
Let $h\in \mod(R)$ be the half twist about $\alpha$ exchanging $\partial_1$ and $\partial_2$ and being the identity on $H$.
Each of the three sets of vertices, $\{ \beta, \gamma, \delta \}$, $\{ \beta, \gamma, h^{-1}(\epsilon)\}$ and $\{ \beta, \gamma, \epsilon\}$, forms a triangle of $L$.

\begin{figure}
\includegraphics[width=10cm]{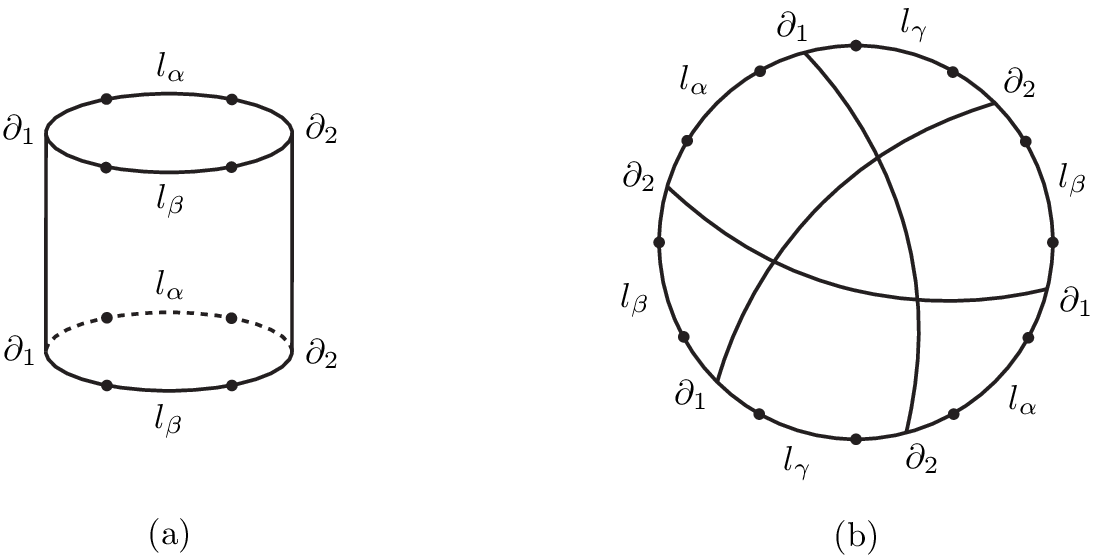}
\caption{}\label{fig-ann}
\end{figure}

We show that there exist at most three triangles of $L$ containing $\{ \beta, \gamma \}$.
If we cut $R$ along the arcs $l_{\alpha}$ and $l_{\beta}$, then we obtain the annulus $A$ whose boundary can be described as in Figure \ref{fig-ann} (a) because $R$ is orientable. 
The arc $l_{\gamma}$ is then given by an arc in $A$ connecting a point of an arc corresponding to $\partial_1$ with a point of an arc corresponding to $\partial_2$. 
This arc in $A$ connects two points in distinct components of $\partial A$ because otherwise $l_{\gamma}$ would be isotopic to either $l_{\alpha}$ or $l_{\beta}$. 
If we cut $A$ along $l_{\gamma}$, then we obtain the disk $D$ in Figure \ref{fig-ann} (b), where the order of the symbols on $\partial D$,
\[\partial_1, l_{\alpha}, \partial_2, l_{\beta}, \partial_1, l_{\gamma}, \partial_2,\ldots,\]
may be reversed. 
This depends on the orientations of $A$ and $D$ and on arcs in $\partial A$ corresponding to $\partial_1$ and $\partial_2$ in which the end points of $l_{\gamma}$ lie. 
There exist exactly three arcs in $D$ connecting a point of an arc corresponding to $\partial_1$ with a point of an arc corresponding to $\partial_2$, up to isotopy, as described in Figure \ref{fig-ann} (b). 
It turns out that there exist at most three triangles of $L$ containing the edge $\{ \beta, \gamma \}$. 
\end{proof}

Recall that we have the simplicial map $\pi \colon L\rightarrow \cal{F}$ defined by $\pi(\beta)=c(\alpha, \beta)$ for each $\beta \in V(L)$, where $c(\alpha, \beta)$ is the curve in Figure \ref{fig-nsarc} (a).

\begin{proof}[Proof of Lemma \ref{lem-seq-tri}]
Let $\Delta$ and $\Delta'$ be triangles of $L$. 
The argument in the first paragraph of the proof of Lemma \ref{lem-tri-three} shows that if we pick an edge of $L$ and the three triangles of $L$ containing it, then the image of them via $\pi$ consists of two triangles of $\cal{F}$ sharing an edge. 
Since any two triangles of $\cal{F}$ are chain-connected in $\cal{F}$, there exists a triangle $\Delta''$ of $\pi^{-1}(\pi(\Delta'))$ such that  $\Delta$ and $\Delta''$ are chain-connected in $L$.
We conclude that $\Delta$ and $\Delta'$ are chain-connected in $L$ because any two triangles in $\pi^{-1}(\pi(\Delta'))$ are chain-connected in $\pi^{-1}(\pi(\Delta'))$ as described in Figure \ref{fig-fiber} (b).
\end{proof}


\section{$S_{1, p}$ with $p\geq 3$}\label{sec-g1}

When $S=S_{1, p}$ is a surface with $p\geq 3$, we show that any superinjective map $\phi$ from $\calc_s(S)$ into itself is induced by an element of $\mod^*(S)$. 
The proof relies on induction on $p$.

\subsection{The case $p=3$}\label{subsec-13}

We put $S=S_{1, 3}$.
In this subsection, we show that any superinjective map $\phi \colon \calc_{s}(S)\rightarrow \calc_s(S)$ is surjective. 
Theorem \ref{thm-cs-aut} then implies that $\phi$ is induced by an element of $\mod^*(S)$. 
We first review several facts on $\calc_s(S)$ discussed in \cite{kida}.

We mean by a {\it hexagon} in $\calc_s(S)$ the full subgraph of $\calc_s(S)$ spanned by exactly six vertices $v_1,\ldots, v_6$ with $i(v_j, v_{j+1})=0$, $i(v_j, v_{j+2})\neq 0$ and $i(v_j, v_{j+3})\neq 0$ for each $j$ mod $6$ (see Figure \ref{fig-hex}). 
Any superinjective map $\phi \colon \calc_s(S)\rightarrow \calc_s(S)$ preserves hexagons in $\calc_s(S)$. 
Fundamental properties of hexagons in $\calc_s(S)$ and superinjective maps from $\calc_s(S)$ into itself are stated in the following two propositions.

\begin{prop}[\ci{Theorem 5.2}{kida}]\label{prop-hex}
Let $S=S_{1, 3}$ be a surface. 
Then for any two hexagons $\Pi_1$, $\Pi_2$ in $\calc_s(S)$, there exists an element $f$ of $\pmod(S)$ with $f(\Pi_1)=\Pi_2$.
\end{prop}

\begin{prop}[\ci{Lemma 5.6}{kida}]\label{prop-phi-13}
Let $S=S_{1, 3}$ be a surface. 
Then any superinjective map from $\calc_s(S)$ into itself preserves vertices corresponding to h-curves and p-curves in $S$, respectively.
\end{prop}

We note that each separating curve in $S$ is either an h-curve or a p-curve in $S$ and that for each h-curve (resp.\ p-curve) $\alpha$ in $S$, any separating curve in $S$ disjoint from $\alpha$ and non-isotopic to $\alpha$ is a p-curve (resp.\ an h-curve) in $S$.

\begin{thm}\label{thm-13}
Let $S=S_{1, 3}$ be a surface. 
Then any superinjective map from $\calc_s(S)$ into itself is surjective.
\end{thm}

\begin{proof}
Put $S=S_{1, 3}$ and let $\phi \colon \calc_s(S)\rightarrow \calc_s(S)$ be a superinjective map. 
Since $\calc_s(S)$ is connected, it is enough to show that for each $\alpha \in V_s(S)$, the map $\phi_{\alpha}\colon \lk_s(\alpha)\rightarrow \lk_s(\phi(\alpha))$ defined as the restriction of $\phi$ is surjective, where for each $\beta \in V_s(S)$, we denote by $\lk_s(\beta)$ the link of $\beta$ in $\calc_s(S)$. 

\begin{figure}
\includegraphics[width=12cm]{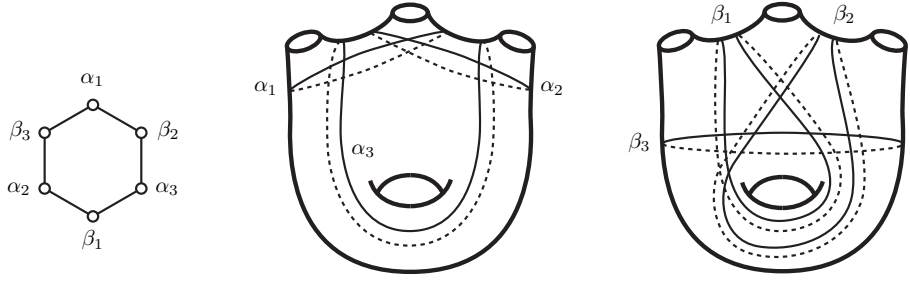}
\caption{A hexagon in $\calc_s(S_{1, 3})$}\label{fig-hex}
\end{figure}

We first assume that $\alpha$ is an h-curve in $S$.
Let $Q_1$ and $Q_2$ denote the components of $S_{\alpha}$ and $S_{\phi(\alpha)}$, respectively, that are homeomorphic to $S_{0, 4}$. 
For any two vertices $\alpha_1$, $\alpha_2$ of $\lk_s(\alpha)$ with $i(\alpha_1, \alpha_2)=2$, we obtain $i(\phi(\alpha_1), \phi(\alpha_2))=2$ by using Proposition \ref{prop-hex} and the fact that $\phi$ preserves hexagons in $\calc_s(S)$.
It follows that $\phi_{\alpha}$ induces an injective simplicial map from the graph $\cal{F}(Q_1)$ into the graph $\cal{F}(Q_2)$ and is thus surjective. 

We next assume that $\alpha$ is a p-curve in $S$.
Let $R_1$ and $R_2$ denote the components of $S_{\alpha}$ and $S_{\phi(\alpha)}$, respectively, that are homeomorphic to $S_{1, 2}$. 
Similarly, Proposition \ref{prop-hex} implies that $\phi_{\alpha}$ induces an injective simplicial map from the graph $\cal{D}(R_1)$ into $\cal{D}(R_2)$ and is thus surjective by Proposition \ref{prop-d}.
\end{proof}

Combining the last theorem with Theorem \ref{thm-cs-aut}, we obtain the following:

\begin{cor}\label{cor-13}
Let $S=S_{1, 3}$ be a surface. 
Then any superinjective map from $\calc_s(S)$ into itself is induced by an element of $\mod^*(S)$.
\end{cor}


\subsection{The case $p\geq 4$}\label{subsec-p4}

Let $S=S_{1, p}$ be a surface with $p\geq 4$ and fix a superinjective map $\phi \colon \calc_s(S)\rightarrow \calc_s(S)$. 
By induction on $p$, we show that $\phi$ is induced by an element of $\mod^*(S)$. 
For each integer $q$ with $2\leq q\leq p$, we mean by a $q$-{\it HBC (hole-bounding curve)} in $S$ a separating curve $\alpha$ in $S$ such that the component of $S_{\alpha}$ of genus zero contains exactly $q$ components of $\partial S$.
Note that each separating curve in $S$ is a $q$-HBC for some integer $q$ with $2\leq q\leq p$.
By Lemma 3.19 of \cite{kida}, for each integer $q$ with $2\leq q\leq p$, the map $\phi$ preserves $q$-HBCs in $S$.

\begin{lem}\label{lem-hbc}
Let $\alpha$ be a $q$-HBC in $S$ with $2\leq q\leq p$. 
Then the map
\[\phi_{\alpha}\colon \lk_{s}(\alpha)\rightarrow \lk_{s}(\phi(\alpha))\]
defined as the restriction of $\phi$ is surjective, where for each $\beta \in V_{s}(S)$, we denote by $\lk_{s}(\beta)$ the link of $\beta$ in $\calc_{s}(S)$.
\end{lem}

\begin{proof}
If $q=2$, then $\lk_s(\alpha)$ is identified with $\calc_s(S_{1, p-1})$, and $\phi_{\alpha}$ is surjective by the hypothesis of the induction.
If $q=p$, then $\lk_s(\alpha)$ is identified with $\calc(S_{0, p+1})$, and $\phi_{\alpha}$ is surjective by Theorem \ref{thm-sha}.

We assume $3\leq q\leq p-1$.
Let $Q$ and $R$ denote the two components of $S_{\alpha}$ with $R$ of genus one, and let $Q_1$ and $R_1$ denote the two components of $S_{\phi(\alpha)}$ with $R_1$ of genus one.
As proved in Lemma 3.19 of \cite{kida}, we have the inclusions
\[\phi(V(Q))\subset V(Q_1)\quad \textrm{and}\quad \phi(V_s(R))\subset V_s(R_1).\]
Choosing an h-curve $\beta$ in $S$ disjoint from $\alpha$ and applying Theorem \ref{thm-sha} to the component of $S_{\beta}$ of genus zero, we obtain the equality $\phi(V(Q))=V(Q_1)$.
Choosing a separating curve $\gamma$ in $Q$ and applying the hypothesis of the induction to the component of $S_{\gamma}$ of genus one, we obtain the equality $\phi(V_s(R))=V_s(R_1)$.
\end{proof}

Lemma \ref{lem-hbc} implies that $\phi$ is surjective because $\calc_s(S)$ is connected.
Combining Theorem \ref{thm-cs-aut}, we obtain the following:

\begin{thm}\label{thm-1p4}
Let $S=S_{1, p}$ be a surface with $p\geq 4$.
Then any superinjective map from $\calc_{s}(S)$ into itself is induced by an element of $\mod^*(S)$.
\end{thm}


\section{Construction of $\Phi$ and its simpliciality}\label{sec-const}

Given a closed surface $S$ with its genus at least three and a superinjective map $\phi \colon \calc_s(S)\rightarrow \calc_s(S)$, Brendle-Margalit \cite{bm} construct a map $\Phi \colon V(S)\rightarrow V(S)$ which coincides with $\phi$ on $V_s(S)$. 
They prove that $\Phi$ defines an automorphism of $\calc(S)$ if $\phi$ is an automorphism of $\calc_s(S)$. 
Their construction can also be applied to the case $S=S_{g, p}$ with $g\geq 2$ and $\left|\chi(S)\right|=2g+p-2\geq 4$ as discussed in \cite{kida}. 
In this section, we review the construction of $\Phi$ and prove simpliciality of $\Phi$ without assuming that $\phi$ is an automorphism.
Sharing pairs defined below play an important role in the construction of $\Phi$.

If $S=S_{g, p}$ is a surface with $g\geq 2$ and $\left|\chi(S)\right|\geq 3$, then for each h-curve (or its isotopy class) $\alpha$ in $S$, we denote by $H_{\alpha}$ the handle cut off by $\alpha$ from $S$, which is naturally identified with a subsurface of $S$.

\begin{defn}\label{defn-share}
Let $S=S_{g, p}$ be a surface with $g\geq 2$ and $\left|\chi(S)\right|\geq 3$.
Let $\alpha, \beta \in V_s(S)$ be h-curves in $S$ and $c\in V(S)$ a non-separating curve in $S$.
We say that $\alpha$ and $\beta$ {\it share} $c$ if there exist representatives $A$, $B$ and $C$ of $\alpha$, $\beta$ and $c$, respectively, such that we have $|A\cap B|=i(\alpha, \beta)$, $H_A\cap H_B$ is an annulus with its core curve $C$, and $S\setminus (H_A\cup H_B)$ is connected.
In this case, we also say that $\{ \alpha, \beta \}$ is a {\it sharing pair} for $c$.
\end{defn}

It is shown that any two sharing pairs in $S$ are sent to each other by an element of $\pmod(S)$. 
Note that when $S$ is a surface of genus less than two, there exists no pair $\{ \alpha, \beta \}$ of h-curves in $S$ satisfying the condition in Definition \ref{defn-share}.

Given a sharing pair $\{ \alpha, \beta \}$ for a non-separating curve $c$ in $S$, one can associate a BP $b(\alpha, \beta)$ in $S$ as follows.
Choosing representatives $A$, $B$ of $\alpha$, $\beta$, respectively, with $|A\cap B|=i(\alpha, \beta)=4$ and choosing a regular neighborhood $N$ of $A\cup B$ in $S$, we define $b(\alpha, \beta)\in \Sigma(S)$ as the set of isotopy classes of boundary components of $N$ which are essential in $S$ and whose isotopy classes are not equal to $c$.
The set $b(\alpha, \beta)$ is in fact a BP in $S$ which cuts off a surface homeomorphic to $S_{1, 2}$ and containing $\alpha$, $\beta$ and $c$.

The following is a summary of properties of superinjective maps from $\calc_s(S)$ into itself which will be needed to construct $\Phi$.

\begin{lem}[\ci{Lemmas 3.18 and 3.19}{kida}]\label{lem-pre-top}
Let $S=S_{g, p}$ be a surface with $g\geq 2$ and $\left|\chi(S)\right|\geq 4$, and let $\phi \colon \calc_s(S)\rightarrow \calc_s(S)$ be a superinjective map. 
Then $\phi$ preserves the topological type of each vertex of $\calc_s(S)$. 
Namely, for each separating curve $\alpha$ in $S$, if $Q_1$ and $Q_2$ denote the components of $S_{\alpha}$ and if $R_1$ and $R_2$ denote the components of $S_{\phi(\alpha)}$, then for each $j=1, 2$,
\begin{itemize}
\item the inclusion $\phi(V_s(Q_j))\subset V_s(R_j)$ holds; and
\item $Q_j$ and $R_j$ are homeomorphic
\end{itemize} 
after exchanging the indices if necessary.
\end{lem}

The following proposition is essentially due to \cite{bm}, where closed surfaces are dealt with (see Section 5.3 in \cite{kida} for the case where a surface has non-empty boundary).

\begin{prop}\label{prop-share}
Let $S=S_{g, p}$ be a surface with $g\geq 2$ and $\left|\chi(S)\right|\geq 4$, and let $\phi \colon \calc_s(S)\rightarrow \calc_s(S)$ be a superinjective map. 
Then the following assertions hold:
\begin{enumerate}
\item The map $\phi$ preserves sharing pairs.
\item Pick a non-separating curve $c$ in $S$ and let $\{ \alpha_1, \beta_1\}$ and $\{ \alpha_2, \beta_2\}$ be sharing pairs for $c$. 
Then $\{ \phi(\alpha_1), \phi(\beta_1)\}$ and $\{ \phi(\alpha_2), \phi(\beta_2)\}$ are sharing pairs for the same non-separating curve in $S$.
\end{enumerate}
\end{prop}

Given a superinjective map $\phi \colon \calc_s(S)\rightarrow \calc_s(S)$, we define a map $\Phi \colon V(S)\rightarrow V(S)$ as follows.
Pick $\alpha \in V(S)$.
If $\alpha$ is separating in $S$, then we set $\Phi(\alpha)=\phi(\alpha)$. 
If $\alpha$ is non-separating in $S$, then we choose a sharing pair $\{ \beta, \gamma \}$ for $\alpha$ and define $\Phi(\alpha)$ to be the non-separating curve shared by the pair $\{ \phi(\beta), \phi(\gamma)\}$.
This is well-defined thanks to Proposition \ref{prop-share}.

\begin{rem}\label{rem-gap}
In Section 4.3 of \cite{bm}, Brendle-Margalit assert that if $S$ is a closed surface of genus more than three and if $\phi \colon \calc_s(S)\rightarrow \calc_s(S)$ is a superinjective map, then the map $\Phi \colon V(S)\rightarrow V(S)$ constructed above defines a superinjective map from $\calc(S)$ into itself.
We point out gaps in their argument to prove superinjectivity of $\Phi$.
(We notice that $\phi$ and $\Phi$ are denoted by $\phi_{\star}$ and $\hat{\phi}_{\star}$, respectively, in \cite{bm}.)
To prove that for any $\alpha, \beta \in V(S)$, we have $i(\alpha, \beta)=0$ if and only if $i(\Phi(\alpha), \Phi(\beta))=0$, Brendle-Margalit make the following three steps:
\begin{enumerate}
\item[(1)] When both $\alpha$ and $\beta$ are separating in $S$, the desired equivalence for $\alpha$ and $\beta$ follows because we have $\Phi =\phi$ on $V_s(S)$ and $\phi$ is superinjective.  
\item[(2)] When both $\alpha$ and $\beta$ are non-separating in $S$, Brendle-Margalit claim that $\alpha$ and $\beta$ are disjoint if and only if there exist sharing pairs $\{ a_1, a_2\}$ for $\alpha$ and $\{ b_1, b_2\}$ for $\beta$ with $i(a_j, b_k)=0$ for any $j, k=1, 2$.
They assert that the desired equivalence for $\alpha$ and $\beta$ follows from this claim.
\item[(3)] When $\alpha$ is separating in $S$ and $\beta$ is non-separating in $S$, Brendle-Margalit claim that $\alpha$ and $\beta$ are disjoint if and only if either $\alpha$ is a part of a sharing pair for $\beta$ or there exists a sharing pair for $\beta$ whose curves are disjoint from $\alpha$.
They assert that the desired equivalence for $\alpha$ and $\beta$ follows from this claim.
\end{enumerate}

First, we point out that the claim in (2) is not correct.
This is because if $\alpha$ and $\beta$ are non-separating curves in $S$ and if $a$ and $b$ are disjoint and non-isotopic h-curves in $S$ with $\alpha \in V(H_a)$ and $\beta \in V(H_b)$, then the surface obtained by cutting $S$ along $\alpha$ and $\beta$ is connected and thus $\{ \alpha, \beta \}$ is not a BP in $S$.
It follows that if $\{ \alpha, \beta \}$ is a BP in $S$, then there exist no sharing pairs $\{ a_1, a_2\}$ for $\alpha$ and $\{ b_1, b_2\}$ for $\beta$ with $i(a_j, b_k)=0$ for any $j, k=1, 2$.
The claim in (2) can be modified as follows.

\begin{lem}\label{lem-dis-ns}
Let $S=S_{g, p}$ be a surface with $g\geq 2$ and $|\chi(S)|\geq 3$.
Let $\alpha$ and $\beta$ be non-separating curves in $S$ which are non-isotopic.
Then we have $i(\alpha, \beta)=0$ and $\{ \alpha, \beta \}$ is not a BP in $S$ if and only if there exist non-isotopic and disjoint h-curves $a$, $b$ in $S$ with $\alpha \in V(H_a)$ and $\beta \in V(H_b)$.
\end{lem}

\begin{proof}
The ``if" part follows because $H_a$ and $H_b$ are disjoint when they are identified with their image via the natural inclusion into $S$.
If $i(\alpha, \beta)=0$ and $\{ \alpha, \beta \}$ is not a BP in $S$, then the surface $Q$ obtained by cutting $S$ along $\alpha$ and $\beta$ is homeomorphic to $S_{g-2, p+4}$. 
Choose p-curves $a$, $b$ in $Q$ such that $i(a, b)=0$ and the pair of pants cut off by $a$ (resp.\ $b$) from $Q$ contains the two components of $\partial Q$ corresponding to $\alpha$ (resp.\ $\beta$).
The curves $a$ and $b$ are h-curves in $S$ via the inclusion of $V(Q)$ into $V(S)$, which cut off handles from $S$ containing $\alpha$ and $\beta$, respectively.
\end{proof}

By the definition of $\Phi$, if $\gamma$ is a non-separating curve in $S$ and $c$ is an h-curve in $S$ with $\gamma \in V(H_c)$, then we have $\Phi(\gamma)\in V(H_{\phi(c)})$.
Using this fact and Lemma \ref{lem-dis-ns}, one can directly show that if $\alpha$ and $\beta$ are disjoint non-separating curves in $S$ such that $\{ \alpha, \beta \}$ is not a BP in $S$, then $\Phi(\alpha)$ and $\Phi(\beta)$ are disjoint.

Secondly, the claim in (3) does not immediately imply that for any separating curve $\alpha$ in $S$ and any non-separating curve $\beta$ in $S$ with $i(\Phi(\alpha), \Phi(\beta))=0$, we have $i(\alpha, \beta)=0$.
This is because we do not assume surjectivity of $\phi$.

In conclusion, to fill these gaps, we need to show that
\begin{enumerate}
\item[(a)] if $\{ \alpha, \beta \}$ is a BP in $S$, then $i(\Phi(\alpha), \Phi(\beta))=0$; and
\item[(b)] if $\alpha$ and $\beta$ are curves in $S$ with $i(\Phi(\alpha), \Phi(\beta))=0$ and if $\beta$ is non-separating in $S$, then $i(\alpha, \beta)=0$.
\end{enumerate}
We will prove assertion (a) in Lemma \ref{lem-simp} and also prove that $\{ \Phi(\alpha), \Phi(\beta)\}$ is a BP in $S$ for each BP $\{ \alpha, \beta \}$ in $S$ by using facts on the graph $\cal{D}$ shown in Section \ref{sec-d}.
Although we do not prove assertion (b) directly, we show that $\Phi$ is induced by an element of $\mod^*(S)$ by proving surjectivity of $\phi$.

If $\phi$ is an automorphism of $\calc_s(S)$, then $\Phi$ is a bijection from $V(S)$ into itself and the map from $V(S)$ into itself associated to $\phi^{-1}$ is equal to $\Phi^{-1}$.
In this case, we can show simpliciality of $\Phi$ (and thus that of $\Phi^{-1}$) without effort as precisely discussed in the proof of Theorem 5.18 of \cite{kida}.
Brendle-Margalit's proof of their Main Theorem 1 in \cite{bm} and Theorem 1 in \cite{bm-add}, stating the natural isomorphism between $\mod^*(S)$ and the abstract commensurator of $\calk(S)$ when $S$ is a closed surface of genus at least three, is therefore valid.
\end{rem}

We now prove simpliciality of $\Phi$ in the following:

\begin{lem}\label{lem-simp}
Let $S=S_{g, p}$ be a surface with $g\geq 2$ and $\left|\chi(S)\right|\geq 4$, and let $\phi \colon \calc_s(S)\rightarrow \calc_s(S)$ be a superinjective map. 
Then the map $\Phi \colon V(S)\rightarrow V(S)$ constructed right after Proposition \ref{prop-share} defines a simplicial map from $\calc(S)$ into itself. 
\end{lem}

Before proving this lemma, we make a brief observation on the set $A(H)$ of isotopy classes of essential simple arcs in a handle $H$, defined in Section \ref{subsec-term}.
For each $l\in A(H)$ and each $a\in V(H)$, we define $i(l, a)$ to be the minimal cardinality of the intersection of representatives of $l$ and $a$.

\begin{lem}\label{lem-dis-arc}
Let $H$ be a handle and choose two curves $a$, $c$ in $H$ with $i(a, c)=1$. 
Then for each $l\in A(H)$, we have either $i(l, a)=0$ or $i(l, c)=0$ if and only if we have $i(l, t_c(a))=i(l, t_c^{-1}(a))=1$.
\end{lem}

\begin{proof} 
There is a one-to-one correspondence between elements of $V(H)$ and of $A(H)$. 
Namely, for each $l\in A(H)$, there exists a unique element $c(l)\in V(H)$ with $i(l, c(l))=0$, and vice versa.
A representative of $c(l)$ is obtained as a boundary component of a regular neighborhood of the union of $\partial H$ and a representative of $l$ in $H$.
Note that for each $l\in A(H)$ and each $c\in V(H)$, we have $i(l, c)=1$ if and only if we have $i(c(l), c)=1$.

Each of $\{ a, c, t_c(a)\}$ and $\{ a, c, t_c^{-1}(a)\}$ forms a triangle in the graph $\cal{F}(H)$.
Since $a$ and $c$ are the only vertices adjacent to both $t_c(a)$ and $t_c^{-1}(a)$ in $\cal{F}(H)$, for each $b\in V(H)$, we have $i(b, t_c(a))=i(b, t_c^{-1}(a))=1$ if and only if $b$ is equal to either $a$ or $c$.
The claim thus follows.
\end{proof}

\begin{proof}[Proof of Lemma \ref{lem-simp}]
It follows from the definition of $\Phi$ that in general, if $\alpha$ is an h-curve in $S$ and $c$ is a non-separating curve in $H_{\alpha}$, then $\Phi(\alpha)$ is also an h-curve in $S$, and $\Phi(c)$ is a curve in the handle $H_{\Phi(\alpha)}$.

Let $\alpha$ and $\beta$ be disjoint curves in $S$. 
If both $\alpha$ and $\beta$ are separating in $S$, then $\Phi(\alpha)$ and $\Phi(\beta)$ are disjoint since $\phi$ is simplicial. 
If $\alpha$ is separating in $S$ and $\beta$ is non-separating in $S$, then there exists an h-curve $\gamma$ in $S$ with $i(\gamma, \alpha)=0$ and $\beta \in V(H_{\gamma})$. 
Since $\alpha$ is either equal to $\gamma$ or in the component of $S_{\gamma}$ that is not a handle, the curves $\Phi(\alpha)$ and $\Phi(\beta)$ are disjoint.

Finally, we suppose that $\alpha$ and $\beta$ are both non-separating in $S$ and non-isotopic. 
If there exist non-isotopic and disjoint h-curves $\gamma$ and $\delta$ in $S$ with $\alpha \in V(H_{\gamma})$ and $\beta \in V(H_{\delta})$, then $\Phi(\alpha)$ and $\Phi(\beta)$ are disjoint because $H_{\phi(\gamma)}$ and $H_{\phi(\delta)}$ are disjoint and we have $\Phi(\alpha)\in V(H_{\phi(\gamma)})$ and $\Phi(\beta)\in V(H_{\phi(\delta)})$. 
Otherwise $\alpha$ and $\beta$ form a BP in $S$ by Lemma \ref{lem-dis-ns}. 
After proving the following two claims, we show that $\Phi(\alpha)$ and $\Phi(\beta)$ are disjoint in this case.

\begin{claim}\label{claim-pre-r}
Let $\{ \alpha, \beta \}$ be a sharing pair in $S$. 
We denote by $R$ the surface cut off by the BP $b(\alpha, \beta)$ from $S$ and containing $\alpha$ and $\beta$.
Similarly, we denote by $\phi(R)$ the surface cut off by the BP $b(\phi(\alpha), \phi(\beta))$ from $S$ and containing $\phi(\alpha)$ and $\phi(\beta)$.
Then we have the inclusion $\phi(V_s(R))\subset V_s(\phi(R))$.
\end{claim}

\begin{proof}
Note that each of $R$ and $\phi(R)$ is homeomorphic to $S_{1, 2}$.
Choose a separating curve $\gamma$ in $S$ cutting off a surface which contains $R$ and is homeomorphic to $S_{2, 1}$. 
We pick a separating curve $\delta$ in $S$ with $i(\alpha, \delta)=i(\beta, \delta)=0$ and $i(\gamma, \delta)\neq 0$.
By Lemma \ref{lem-pre-top}, $\phi(\gamma)$ cuts off from $S$ a surface homeomorphic to $S_{2, 1}$ and containing $\phi(R)$.
Superinjectivity of $\phi$ implies that $\phi(\delta)$ is disjoint from $\phi(\alpha)$ and $\phi(\beta)$ and intersects $\phi(\gamma)$.
It follows that if $C$ and $D$ are representatives of $\phi(\gamma)$ and $\phi(\delta)$, respectively, with $|C\cap D|=i(\phi(\gamma), \phi(\delta))$, then the two curves in $b(\phi(\alpha), \phi(\beta))$ are boundary components of a regular neighborhood of $C\cup D$ in $S$.
If a separating curve $\epsilon$ in $S$ satisfies $i(\phi(\gamma), \epsilon)=i(\phi(\delta), \epsilon)=0$ and either $i(\phi(\alpha), \epsilon)\neq 0$ or $i(\phi(\beta), \epsilon)\neq 0$, then $\epsilon$ is a curve in $\phi(R)$. 
The claim thus follows.
\end{proof}

\begin{claim}\label{claim-handle-bij}
For each h-curve $\alpha$ in $S$, the restriction of $\Phi$ to $V(H_{\alpha})$ induces an isomorphism between the graphs $\cal{F}(H_{\alpha})$ and $\cal{F}(H_{\phi(\alpha)})$.
\end{claim}

\begin{proof}
Choose an h-curve $\beta_{0}$ in $S$ such that $\{ \alpha, \beta_0\}$ is a sharing pair in $S$. 
To prove the claim, we may assume $\phi(\alpha)=\alpha$ and $\phi(\beta_0)=\beta_0$. 
Let $R$ denote the surface cut off by the BP $b(\alpha, \beta_0)$ from $S$ and containing $\alpha$ and $\beta_0$. 
Proposition \ref{prop-share} and Claim \ref{claim-pre-r} show that $\phi$ induces an injective simplicial map from $\cal{D}=\cal{D}(R)$ into itself, which is an automorphism of $\cal{D}$ by Proposition \ref{prop-d}. 
In particular, $\phi$ induces an automorphism of $L$, the link of $\alpha$ in $\cal{D}$. 
Put $\cal{F}=\cal{F}(H_{\alpha})$ and let $\pi \colon L\rightarrow \cal{F}$ be the simplicial map defined in Section \ref{subsec-fiber}.

We now show that for any two curves $b, c\in V(H_{\alpha})$ with $i(b, c)=1$, the equality $i(\Phi(b), \Phi(c))=1$ holds, that is, $\Phi$ preserves edges of $\cal{F}$. 
We choose an edge $\{ \beta, \gamma \}$ of $L$ with $\pi(\beta)=b$ and $\pi(\gamma)=c$.
Since $\phi$ induces an automorphism of $L$, the two vertices $\phi(\beta)$ and $\phi(\gamma)$ form an edge of $L$.
Since the fiber of $\pi$ over each vertex of $\cal{F}$ is zero-dimensional by Lemma \ref{lem-d-line}, the two vertices $\pi(\phi(\beta))$ and $\pi(\phi(\gamma))$ are distinct and thus form an edge of $\cal{F}$. 
Since we have $\pi(\phi(\beta))=\Phi(b)$ and $\pi(\phi(\gamma))=\Phi(c)$ by the definition of $\Phi$, we obtain $i(\Phi(b), \Phi(c))=1$.

Proposition \ref{prop-share} shows that for each vertex $d$ of $\cal{F}$, the inclusion $\phi(\pi^{-1}(d))\subset \pi^{-1}(\Phi(d))$ holds. 
Since the fiber of $\pi$ over an edge of $\cal{F}$ is a bi-infinite line by Lemma \ref{lem-d-line} and since $\phi$ is injective, the equality $\phi(\pi^{-1}(\{ b, c\}))= \pi^{-1}(\Phi(\{ b, c\}))$ holds for each edge $\{ b, c\}$ of $\cal{F}$.
We thus have $\phi(\pi^{-1}(d))=\pi^{-1}(\Phi(d))$ for each vertex $d$ of $\cal{F}$.
Injectivity of $\phi$ again implies that $\Phi$ induces an injective simplicial map from $\cal{F}$ into itself and thus an automorphism of $\cal{F}$.
\end{proof}

\begin{claim}\label{claim-bp-dis}
If $a$ and $b$ are non-separating curves in $S$ with $\{ a, b\}$ a BP in $S$, then we have $\Phi(a)\neq \Phi(b)$ and $i(\Phi(a), \Phi(b))=0$.
\end{claim}

\begin{proof}
When two non-separating curves $d$ and $e$ in $S$ satisfy the equality $i(d, e)=1$, let us write $d \bot e$ for simplicity. 

Choose a non-separating curve $c$ in $S$ with $a \bot c$ and $b \bot c$.
We denote by $H$ the handle filled by $a$ and $c$.
If $A$ and $C$ are representatives of $a$ and $c$, respectively, with $|A\cap C|=i(a, c)=1$, then $H$ is obtained as a regular neighborhood of $A\cup C$.
Let $\alpha$ denote the boundary curve of $H$.
Similarly, we denote by $K$ the handle filled by $b$ and $c$ and denote by $\beta$ the boundary curve of $K$.
Let $\phi(H)$ and $\phi(K)$ denote the handles cut off by $\phi(\alpha)$ and $\phi(\beta)$ from $S$, respectively.
The handle $\phi(H)$ is filled by $\Phi(a)$ and $\Phi(c)$ because we have $\Phi(a)\bot \Phi(c)$ by Claim \ref{claim-handle-bij}.
Similarly, the handle $\phi(K)$ is filled by $\Phi(b)$ and $\Phi(c)$ because we have $\Phi(b)\bot \Phi(c)$ by Claim \ref{claim-handle-bij}.
It follows from $\phi(\alpha)\neq \phi(\beta)$ that we have $\Phi(a)\neq \Phi(b)$.

We set
\[U=\{ \, d \in V(H) \mid d \bot c \, \} =\{ \, t_{c}^{n}(a)\mid n\in \mathbb{Z}\, \}.\]
By Claim \ref{claim-handle-bij}, we have
\[\Phi(U)=\{ \, d \in V(\phi(H))\mid d \bot \Phi(c)\, \}=\{ \, t_{\Phi(c)}^{n}(\Phi(a)) \mid n\in \mathbb{Z}\, \}.\]
The two equalities
\[\{ t_{c}^{\pm 1}(a)\}=\{ \, d \in U\mid d \bot a \, \},\quad \{ t_{\Phi(c)}^{\pm 1}(\Phi(a))\}=\{ \, e \in \Phi(U)\mid e \bot \Phi(a) \, \}\]
imply the equality $\{ \Phi(t_{c}^{\pm 1}(a))\}=\{ t_{\Phi(c)}^{\pm 1}(\Phi(a))\}$.
Claim \ref{claim-handle-bij} then implies
\[\Phi(a)\bot \Phi(c),\quad \Phi(b)\bot \Phi(c),\quad \Phi(b)\bot t_{\Phi(c)}^{\pm 1}(\Phi(a)),\]
where the third relation follows from $b \bot t_c^{\pm 1}(a)$. 
The first and second relations show that $\Phi(b)\cap \phi(H)$ consists of an essential simple arc $l$ in $\phi(H)$ intersecting $\Phi(c)$ once and essential simple arcs in $\phi(H)$ which are disjoint from $\Phi(c)$ and mutually isotopic. 
If there were a component $r$ of $\Phi(b)\cap \phi(H)$ disjoint from $\Phi(c)$, then $r$ would intersect $t_{\Phi(c)}^{\pm 1}(\Phi(a))$ once, respectively, because we have $\Phi(c)\bot t_{\Phi(c)}^{\pm 1}(\Phi(a))$. 
The third relation then implies that $l$ does not intersect $t_{\Phi(c)}^{\pm 1}(\Phi(a))$. 
This is impossible because a curve in $\phi(H)$ disjoint from $l$ uniquely exists up to isotopy. 
We thus proved that $\Phi(b)\cap \phi(H)$ consists of only $l$. 
Since $l$ intersects $\Phi(c)$ and $t_{\Phi(c)}^{\pm 1}(\Phi(a))$ once, respectively, 
Lemma \ref{lem-dis-arc} implies that $l$ is disjoint from $\Phi(a)$.
We therefore conclude that $\Phi(b)$ is disjoint from $\Phi(a)$.
\end{proof}

As discussed before Claim \ref{claim-pre-r}, Claim \ref{claim-bp-dis} completes the proof of Lemma \ref{lem-simp}.
\end{proof}

The following fact will be used in Section \ref{sec-30}.

\begin{lem}\label{lem-bp-bp}
In the notation of Lemma \ref{lem-simp}, the map $\Phi$ preserves BPs in $S$.
That is, if $\{ a, b\}$ is a BP in $S$, then so is $\{ \Phi(a), \Phi(b)\}$.
\end{lem}

\begin{proof}
Suppose that there exists a BP $\{ a, b\}$ in $S$ such that $\{ \Phi(a), \Phi(b)\}$ is not a BP in $S$.
By Claim \ref{claim-bp-dis}, the surface $Q$ obtained by cutting $S$ along $\Phi(a)$ and $\Phi(b)$ is homeomorphic to $S_{g-2, p+4}$.  
On the other hand, there exists a simplex $\sigma$ of $\calc_s(S)$ consisting of $g-1$ h-curves in $S$ disjoint from $a$ and $b$.
Choose h-curves $\delta$ and $\epsilon$ in $S$ with $a\in V(H_{\delta})$, $b\in V(H_{\epsilon})$ and $i(\delta, \sigma)=i(\epsilon, \sigma)=0$.
For each $\gamma \in \sigma$, both $\Phi(a)$ and $\Phi(b)$ are curves in the component of $S_{\phi(\gamma)}$ that is not a handle since we have $\Phi(a)\in V(H_{\phi(\delta)})$ and $\Phi(b)\in V(H_{\phi(\epsilon)})$.
It follows that $\phi(\gamma)$ is an h-curve in $Q$ for each $\gamma \in \sigma$.
This is a contradiction because any collection of pairwise non-isotopic and disjoint h-curves in $Q$ consists of at most $g-2$ curves.
\end{proof}


\section{$S_{2, 2}$}\label{sec-22}

We put $S=S_{2, 2}$ and fix a superinjective map $\phi \colon \calc_s(S)\rightarrow \calc_s(S)$ throughout this section. 
We denote by $\Phi \colon \calc(S)\rightarrow \calc(S)$ the simplicial map extending $\phi$, constructed right after Proposition \ref{prop-share}. 
For each non-separating curve $c$ in $S$, let
\[\Phi_c\colon \lk(c)\cap \calc_s(S)\rightarrow \lk(\Phi(c))\cap \calc_s(S)\] 
be the simplicial map defined as the restriction of $\Phi$, where for each $\alpha \in V(S)$, we denote by $\lk(\alpha)$ the link of $\alpha$ in $\calc(S)$.
In Lemma \ref{lem-Phi-c-surj}, we will prove that $\Phi_c$ is surjective for each $c$. 
Once this lemma is shown, we can readily prove that $\Phi$ is injective and is therefore an automorphism of $\calc(S)$ by Theorem \ref{thm-sha} (see the proof of Theorem \ref{thm-22} for a precise argument). 
A large part of this section is thus devoted to proving surjectivity of $\Phi_c$.

We fix a non-separating curve $c$ in $S$ and may assume $\Phi(c)=c$ until Lemma \ref{lem-Phi-c-surj} to prove surjectivity of $\Phi_c$.
Let $\partial_1$ and $\partial_2$ denote the boundary components of $S_c$ corresponding to $c$.
We first introduce a simplicial graph associated to $c$.

\medskip

\noindent {\bf Graph $\cal{E}$.} We define the simplicial graph $\cal{E}$ as follows.
The set of vertices of $\cal{E}$, denoted by $V(\cal{E})$, is defined as the set of all elements of $V_s(S)$ corresponding to an h-curve $\alpha$ in $S$ such that $c$ is a curve in the handle cut off by $\alpha$ from $S$. 
Two vertices of $\cal{E}$ are connected by an edge of $\cal{E}$ if and only if the two h-curves corresponding to them form a sharing pair for $c$ in $S$.

\medskip

For each $\alpha \in V(\cal{E})$, we denote by $\lk_{\cal{E}}(\alpha)$ the link of $\alpha$ in $\cal{E}$ and denote by $V(\lk_{\cal{E}}(\alpha))$ the set of vertices of $\lk_{\cal{E}}(\alpha)$.

\begin{lem}\label{lem-p-conn}
The graph $\cal{E}$ is connected.
\end{lem}

\begin{proof}
We note that $V(\cal{E})$ is naturally identified with the subset of $V(S_c)$ consisting of all elements corresponding to a p-curve in $S_c$ cutting off a pair of pants containing $\partial_1$ and $\partial_2$.
Let $\alpha$ be the curve in Figure \ref{fig-arc} (a).
We define $T$ as the set consisting of the Dehn twists about the curves in Figure \ref{fig-arc} (b) and their inverses.
The group $\pmod(S_c)$ is generated by $T$ (see \cite{gervais}).
Since for each $h\in T$, either $h\alpha =\alpha$ or $h\alpha$ and $\alpha$ are connected by an edge of $\cal{E}$ and since any two vertices of $\cal{E}$ are sent to each other by an element of $\pmod(S_c)$, connectivity of $\cal{E}$ can be proved as in the proof of Lemma \ref{lem-d-conn}.  
\end{proof}

We next introduce a set of arcs as follows.
 
\medskip

\noindent {\bf Set $A_{\alpha}$.} Pick $\alpha \in V(\cal{E})$ and let $\Sigma_{\alpha}$ denote the component of $S_{\alpha}$ that is not a handle. 
We define $A_{\alpha}$ to be the subset of $A(\Sigma_{\alpha})$ consisting of all elements whose representatives are non-separating in $\Sigma_{\alpha}$ and connect two distinct points of the boundary component of $\Sigma_{\alpha}$ corresponding to $\alpha$.

\medskip

For each edge $\{ \alpha, \beta \}$ of $\cal{E}$, we define an element $r_{\alpha}(\beta)$ of $A_{\alpha}$ as follows.
Let $b(\alpha, \beta)$ be the BP in $S$ associated with the sharing pair $\{ \alpha, \beta \}$ in $S$, defined right after Definition \ref{defn-share}.
Since the BP $b(\alpha, \beta)$ cuts off a pair of pants from $\Sigma_{\alpha}$, we have an essential simple arc in $\Sigma_{\alpha}$ connecting two distinct points of the boundary component of $\Sigma_{\alpha}$ corresponding to $\alpha$, which uniquely exists up to isotopy.
We denote by $r_{\alpha}(\beta)$ the isotopy class of that essential simple arc in $\Sigma_{\alpha}$.

\begin{figure}
\includegraphics[width=12.5cm]{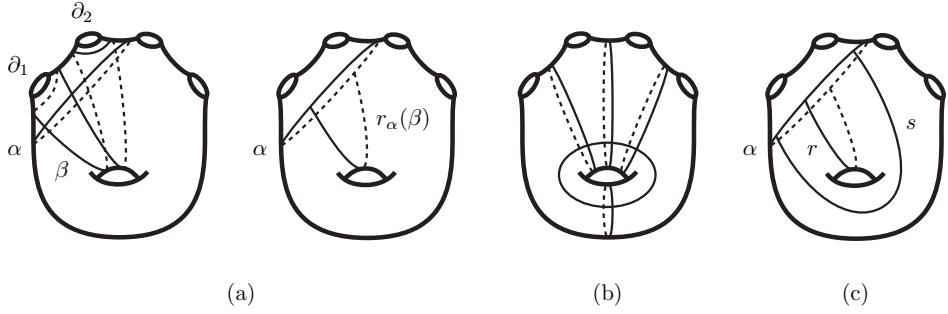}
\caption{The surface obtained by cutting $S$ along $c$ is described. The pair $\{ \alpha, \beta \}$ is an edge of the graph $\cal{E}$.}\label{fig-arc}
\end{figure}

The element $r_{\alpha}(\beta)$ can also be characterized in the following way.
Let $\{ \alpha, \beta \}$ be an edge of $\cal{E}$, and choose representatives $A$ and $B$ of $\alpha$ and $\beta$, respectively, with $|A\cap B|=i(\alpha, \beta)$.
We denote by $\Sigma_A$ the component of the surface obtained by cutting $S$ along $A$ that is not a handle.
The intersection $B\cap \Sigma_A$ consists of exactly two essential simple arcs in $\Sigma_A$ whose isotopy classes are equal to $r_{\alpha}(\beta)$.

\begin{lem}\label{lem-r-image}
For each curve $\alpha \in V(\cal{E})$ and each arc $r\in A_{\phi(\alpha)}$, there exists a curve $\beta \in V(\lk_{\cal{E}}(\alpha))$ satisfying the equality $r_{\phi(\alpha)}(\phi(\beta))=r$.
\end{lem}

\begin{proof}
Let $\phi_{\alpha}\colon \calc_s(\Sigma_{\alpha})\rightarrow \calc_s(\Sigma_{\phi(\alpha)})$ be the map defined as the restriction of $\phi$. 
Corollary \ref{cor-13} shows that $\phi_{\alpha}$ is induced by a homeomorphism from $\Sigma_{\alpha}$ into $\Sigma_{\phi(\alpha)}$, which sends $\alpha$ to $\phi(\alpha)$. 
Let $W$ be the set of all elements of $V_s(\Sigma_{\phi(\alpha)})$ disjoint from $r$.
Note that $r$ is the only element of $A_{\phi(\alpha)}$ disjoint from all elements of $W$.
There exists a unique element $q\in A_{\alpha}$ such that $\phi^{-1}_{\alpha}(W)$ is equal to the set of all elements of $V_s(\Sigma_{\alpha})$ disjoint from $q$. 
Choose $\beta \in V(\lk_{\cal{E}}(\alpha))$ with $r_{\alpha}(\beta)=q$. 
Since each element of $\phi_{\alpha}^{-1}(W)$ is disjoint from $\beta$, each element of $W$ is disjoint from $\phi(\beta)$. 
We then have the equality $r_{\phi(\alpha)}(\phi(\beta))=r$.
\end{proof}

By Corollary \ref{cor-13}, for each $\alpha \in V(\cal{E})$, the restriction of $\phi$ to $\calc_s(\Sigma_{\alpha})$ is induced by a homeomorphism from $\Sigma_{\alpha}$ onto $\Sigma_{\phi(\alpha)}$ sending $\alpha$ to $\phi(\alpha)$.
We thus have the induced bijection $\Phi_{\alpha}\colon A_{\alpha}\rightarrow A_{\phi(\alpha)}$.

\begin{lem}\label{lem-r-eq}
Pick $\alpha \in V(\cal{E})$ and $r\in A_{\alpha}$, and set
\[B=\{ \, \beta \in V(\lk_{\cal{E}}(\alpha))\mid r_{\alpha}(\beta)=r\, \}.\]
Then we have the equality
\[\phi(B)=\{ \, \delta \in V(\lk_{\cal{E}}(\phi(\alpha)))\mid r_{\phi(\alpha)}(\delta)=\Phi_{\alpha}(r)\, \}.\]
\end{lem}

\begin{proof}
By using the set of all elements of $V_s(\Sigma_{\alpha})$ disjoint from $r$ as in the proof of Lemma \ref{lem-r-image}, we can show that the left hand side is contained in the right hand side in the desired equality.

Let $s$ be an element of $A_{\alpha}$ such that $s$ is disjoint and distinct from $r$, and the end points of disjoint representatives of $r$ and $s$ appear alternatively along $\alpha$ (see Figure \ref{fig-arc} (c)).
Let $h\in \mod(S_c)$ be the half twist about $\alpha$ exchanging $\partial_1$ and $\partial_2$ and being the identity on $\Sigma_{\alpha}$.
We set
\[\Gamma =\{ \, \gamma \in V(\lk_{\cal{E}}(\alpha))\mid r_{\alpha}(\gamma)=s\, \}.\]
Applying the argument in the proof of Lemma \ref{lem-d-line}, we have a numbering of elements, $B=\{ \beta_n\}_{n\in \mathbb{Z}}$ and $\Gamma =\{ \gamma_m\}_{m\in \mathbb{Z}}$, such that
\begin{itemize}
\item $h(\beta_n)=\beta_{n+1}$ and $h(\gamma_m)=\gamma_{m+1}$ for any $n, m\in \mathbb{Z}$; and
\item the full subgraph of $\cal{E}$ spanned by $B\cup \Gamma$ is the bi-infinite line with $\beta_n$ adjacent to $\gamma_n$ and $\gamma_{n+1}$ for each $n\in \mathbb{Z}$.
\end{itemize}
We also have the inclusions
\begin{align*}
\phi(B)&\subset \{ \, \delta \in V(\lk_{\cal{E}}(\phi(\alpha)))\mid r_{\phi(\alpha)}(\delta)=\Phi_{\alpha}(r)\, \},\\
\phi(\Gamma)&\subset \{ \, \epsilon \in V(\lk_{\cal{E}}(\phi(\alpha)))\mid r_{\phi(\alpha)}(\epsilon)=\Phi_{\alpha}(s)\, \}.
\end{align*}
Since the map $\Phi_{\alpha}\colon A_{\alpha}\rightarrow A_{\phi(\alpha)}$ is induced by a homeomorphism from $\Sigma_{\alpha}$ onto $\Sigma_{\phi(\alpha)}$ sending $\alpha$ to $\phi(\alpha)$, the two elements $\Phi_{\alpha}(r)$ and $\Phi_{\alpha}(s)$ are disjoint and distinct, and the end points of disjoint representatives of $\Phi_{\alpha}(r)$ and $\Phi_{\alpha}(s)$ appear alternatively along $\phi(\alpha)$. 
The argument in the proof of Lemma \ref{lem-d-line} shows that the subgraph of $\cal{E}$ spanned by the union of the right hand sides of the above two inclusions is thus a bi-infinite line. 
Injectivity of $\phi$ implies that both of the converse inclusions hold. 
The lemma follows.
\end{proof}

\begin{lem}\label{lem-Phi-c-surj}
If $\Phi(c)=c$, then the map
\[\Phi_c\colon \lk(c)\cap \calc_s(S)\rightarrow \lk(c)\cap \calc_s(S)\]
defined as the restriction of $\Phi$ is surjective.
\end{lem}

\begin{proof}
Since $\phi$ preserves sharing pairs for $c$, $\phi$ induces a simplicial map $\phi_c\colon \cal{E}\rightarrow \cal{E}$. 
Lemmas \ref{lem-r-image} and \ref{lem-r-eq} show that for each $\alpha \in V(\cal{E})$, the map from $\lk_{\cal{E}}(\alpha)$ into $\lk_{\cal{E}}(\phi(\alpha))$ induced by $\phi_c$ is surjective. 
It follows from Lemma \ref{lem-p-conn} that the map $\phi_c\colon \cal{E}\rightarrow \cal{E}$ is a simplicial automorphism. 
In particular, the image of $\Phi_c$ contains all h-curves $\alpha$ in $S$ with $\Phi(c)=c\in V(H_{\alpha})$.

Let $\beta \in V(S_c)\cap V_s(S)$ be a curve which is not an h-curve in $S$ cutting off a handle containing $c$. 
There then exists an h-curve $\gamma$ in $S$ with $c\in V(H_{\gamma})$ and $i(\gamma, \beta)=0$.
The argument in the previous paragraph shows that there exists a curve $\alpha \in V(\cal{E})$ with $\Phi_c(\alpha)=\gamma$.
Theorem \ref{thm-13} implies that the map $\phi_{\alpha}\colon \calc_s(\Sigma_{\alpha})\rightarrow \calc_s(\Sigma_{\gamma})$ defined as the restriction of $\phi$ is surjective. 
In particular, the image of $\phi_{\alpha}$ contains $\beta$, and so does $\Phi_c$.
\end{proof}

Using the last lemma, we conclude the following:

\begin{thm}\label{thm-22}
Let $S=S_{2, 2}$ be a surface.
Then any superinjective map from $\calc_{s}(S)$ into itself is induced by an element of $\mod^*(S)$.
\end{thm}

\begin{proof}
Let $c$ and $d$ be non-separating curves in $S$ with $\Phi(c)=\Phi(d)$. 
Lemma \ref{lem-Phi-c-surj} shows that the two maps
\begin{align*}
\Phi_{c}&\colon \lk(c)\cap \calc_{s}(S)\rightarrow \lk(\Phi(c))\cap \calc_{s}(S),\\
\Phi_{d}&\colon \lk(d)\cap \calc_{s}(S)\rightarrow \lk(\Phi(d))\cap \calc_{s}(S)
\end{align*}
defined as the restriction of $\Phi$ are surjective, and their images are equal. 
Since these two maps are restrictions of the injective map $\phi$, we obtain the equality $c=d$. 
It follows that $\Phi$ is injective and is thus induced by an element of $\mod^*(S)$ by Theorems \ref{thm-cc} and \ref{thm-sha}.
\end{proof}


\section{$S_{g, p}$ with $g\geq 2$ and $\left|\chi\right|\geq 5$}\label{sec-g2}

Let $S=S_{g, p}$ be a surface with $g\geq 2$ and $\left|\chi(S)\right|=2g+p-2\geq 5$. 
For each superinjective map $\phi \colon \calc_s(S)\rightarrow \calc_s(S)$, we prove that the simplicial map $\Phi \colon \calc(S)\rightarrow \calc(S)$ constructed right after Proposition \ref{prop-share} is induced by an element of $\mod^*(S)$, by induction on the lexicographic order of $(g, p)$.
The following lemma will be used to complete the inductive argument.
We mean by an {\it hp-curve} in $S$ a curve in $S$ which is either an h-curve or a p-curve in $S$.

\begin{lem}\label{lem-hp-conn}
Let $X$ be a surface with its genus at least two and $\left|\chi(X)\right|\geq 4$.
Then the full subcomplex of $\calc_s(X)$ spanned by all vertices corresponding to hp-curves in $X$ is connected.
\end{lem}

\begin{proof}
The idea to prove this lemma is based on Lemma 2.1 of \cite{putman-conn} as in Lemmas \ref{lem-d-conn} and \ref{lem-p-conn}.
It suffices to show that any two vertices of $\calc_s(X)$ corresponding to h-curves in $X$ can be connected by a path in $\calc_s(X)$ consisting of vertices corresponding to hp-curves in $X$ because for any p-curve $a$ in $X$, there exists an h-curve in $X$ disjoint from $a$.

\begin{figure}
\includegraphics[width=12cm]{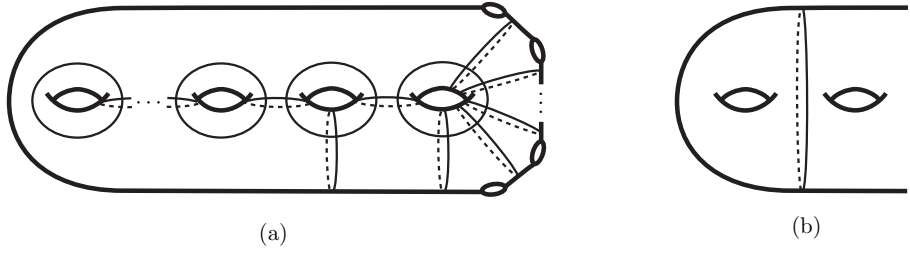}
\caption{If $S$ is a surface of positive genus, then $\pmod(S)$ is generated by the Dehn twists about the curves in (a) (see \cite{gervais}).}\label{fig-mcg}
\end{figure}

We define $T$ to be the set consisting of the Dehn twists about the curves in Figure \ref{fig-mcg} (a) and their inverses.
It is known that $\pmod(X)$ is generated by $T$ (see \cite{gervais}).
Let $\alpha$ denote the h-curve in Figure \ref{fig-mcg} (b).
One can check that for each $h\in T$, either $h\alpha =\alpha$ or there exists an hp-curve $\beta$ in $X$ with $i(h\alpha, \beta)=i(\alpha, \beta)=0$.
Since any two h-curves in $X$ are sent to each other by an element of $\pmod(X)$, the same argument as in Lemma \ref{lem-d-conn} concludes the lemma.
\end{proof}

\begin{thm}\label{thm-g2}
Let $S=S_{g, p}$ be a surface with $g\geq 2$ and $\left|\chi(S)\right|\geq 5$.
Then any superinjective map from $\calc_s(S)$ into itself is induced by an element of $\mod^*(S)$.
\end{thm}

\begin{proof}
If $\alpha$ is an h-curve in $S$, then the component of $S_{\alpha}$ that is not a handle is homeomorphic to $S_{g-1, p+1}$. 
If $\alpha$ is a p-curve in $S$, then $p\geq 2$ and the component of $S_{\alpha}$ that is not a pair of pants is homeomorphic to $S_{g, p-1}$. 
Since we assume $(g, p)\neq (2, 2), (3, 0)$, Theorems \ref{thm-1p4} and \ref{thm-22} and the hypothesis of the induction imply that the map $\phi_{\alpha}\colon \lk_s(\alpha)\rightarrow \lk_s(\phi(\alpha))$ defined as the restriction of $\phi$ is an isomorphism for each hp-curve $\alpha$ in $S$, where $\lk_s(\beta)$ denotes the link of $\beta$ in $\calc_s(S)$ for each $\beta \in V_s(S)$. 
Lemma \ref{lem-hp-conn} implies that $\phi$ is surjective. 
Applying Theorem \ref{thm-cs-aut}, we conclude the theorem.
\end{proof}


\section{$S_{3, 0}$}\label{sec-30}

We put $S=S_{3, 0}$ throughout this section. 
This case is dealt with independently because the component of the surface obtained by cutting $S$ along an h-curve in $S$ is homeomorphic to $S_{2, 1}$ and inductive argument as in Section \ref{sec-g2} cannot be applied.
We first prove that any superinjective map $\psi$ from the Torelli complex $\calt(S)$ into itself is induced by an element of $\mod^*(S)$.

\begin{prop}\label{prop-30-t}
Any superinjective map $\psi \colon \calt(S)\rightarrow \calt(S)$ is induced by an element of $\mod^*(S)$.
\end{prop}

\begin{proof}
By Lemma 3.7 in \cite{kida}, we know that $\psi$ preserves vertices which correspond to separating curves and BPs in $S$, respectively. 
Applying the construction of a simplicial map from $\calc(S)$ into itself, discussed right after Proposition \ref{prop-share}, to the restriction of $\psi$ to $\calc_s(S)$, we obtain a simplicial map $\Psi \colon \calc(S)\rightarrow \calc(S)$.

\begin{claim}\label{claim-bp}
The equality
\[\{ \Psi(b_1), \Psi(b_2)\}=\psi(\{ b_1, b_2\})\]
holds for each BP $\{ b_1, b_2\}$ in $S$.
\end{claim}

\begin{proof}
Pick a BP $\{ b_1, b_2\}$ in $S$.
Let $\alpha_1$ and $\alpha_2$ be the curves in $S$ described in Figure \ref{fig-curves30}.
We note that $\{ \alpha_1, \alpha_2\}$ is a sharing pair in $S$ with $b(\alpha_1, \alpha_2)=\{ b_1, b_2\}$.
In general, for each sharing pair $\{ \alpha, \beta \}$ in $S$, $b(\alpha, \beta)$ is the only BP in $S$ disjoint from $\alpha$ and $\beta$.
By Lemma \ref{lem-bp-bp}, $\{ \Psi(b_1), \Psi(b_2)\}$ is a BP in $S$.
Since $\{ \Psi(b_1), \Psi(b_2)\}$ and $\psi(\{ b_1, b_2\})$ are BPs in $S$ disjoint from the sharing pair $\{ \psi(\alpha_1), \psi(\alpha_2)\}$, we have the desired equality. 
\end{proof}

\begin{figure}
\includegraphics[width=8cm]{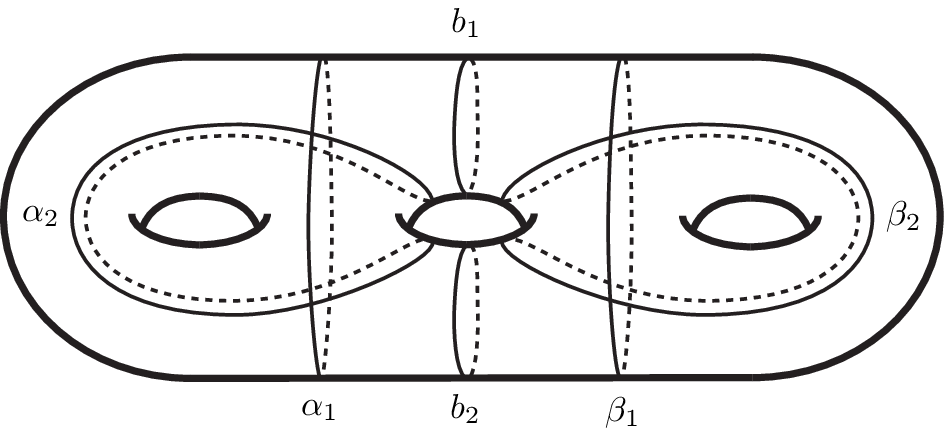}
\caption{}\label{fig-curves30}
\end{figure}

Let $c$ be a non-separating curve in $S$. 
We define a simplicial map $\psi_c\colon \calc_s(S_c)\rightarrow \calc_s(S_{\Psi(c)})$ as follows.
Pick $\alpha \in V_s(S_c)$.
If the curve $\alpha$ is separating in $S$, then we set $\psi_c(\alpha)=\psi(\alpha)$.
Otherwise $\{ \alpha, c\}$ is a BP in $S$ and we have the equality $\psi(\{ \alpha, c\})=\{ \Psi(\alpha), \Psi(c)\}$ by Claim \ref{claim-bp}.
In this case, we set $\psi_c(\alpha)=\Psi(\alpha)$.
Since $\psi \colon \calt(S)\rightarrow \calt(S)$ is superinjective, so is $\psi_c$. 
Theorem \ref{thm-22} shows that $\psi_c\colon \calc_s(S_c)\rightarrow \calc_s(S_{\Psi(c)})$ is an isomorphism.

If $c$ and $d$ are non-separating curves in $S$ with $\Psi(c)=\Psi(d)$, then the images of the two maps $\psi_c$ and $\psi_d$ are equal. 
Since $\psi$ is injective, the equality $\calc_s(S_c)=\calc_s(S_d)$ holds, and we thus have $c=d$. 
It follows that $\Psi$ is injective.
Theorems \ref{thm-cc} and \ref{thm-sha} show that $\Psi$ is induced by an element of $\mod^*(S)$.
\end{proof}

Let $\phi \colon \calc_s(S)\rightarrow \calc_s(S)$ be a superinjective map, and let $\Phi \colon \calc(S)\rightarrow \calc(S)$ be the simplicial map constructed right after Proposition \ref{prop-share}. 
In the rest of this section, we prove that $\Phi$ is an automorphism by using Proposition \ref{prop-30-t}. 
We note that $\Phi$ induces a simplicial map from $\calt(S)$ into itself by Lemma \ref{lem-bp-bp}. 
This induced map is also denoted by the same symbol $\Phi$.

\begin{lem}\label{lem-30-Phi-link}
Let $b$ be a BP in $S$, and let $R_1$ and $R_2$ denote the two components of $S_b$. 
We suppose that the equality $\Phi(b)=b$ holds and that for each $j=1, 2$, the inclusion
\[\Phi(\calc(R_j)\cap \calc_s(S))\subset \calc(R_j)\cap \calc_s(S)\]
holds. 
Then for each $j=1, 2$, the map
\[\Phi_j\colon \calc(R_j)\cap \calc_s(S)\rightarrow \calc(R_j)\cap \calc_s(S)\] defined as the restriction of $\Phi$ is surjective.
\end{lem}

\begin{proof}
For each $j=1, 2$, the map $\Phi_j$ preserves two separating curves in $R_j$ whose intersection number is equal to four since $\phi$ preserves sharing pairs in $S$. 
It follows that $\Phi_j$ induces an injective simplicial map from the graph $\cal{D}=\cal{D}(R_j)$, defined in Section \ref{sec-d}, into itself.
Proposition \ref{prop-d} then shows that $\Phi_j$ is surjective. 
\end{proof}

\begin{lem}
The simplicial map $\Phi \colon \calt(S)\rightarrow \calt(S)$ is superinjective.
\end{lem}

\begin{proof}
We first prove that if $a$ is a separating curve in $S$ and $b=\{ b_{1}, b_{2}\}$ is a BP in $S$ with $i(a, b)\neq 0$, then $i(\Phi(a), \Phi(b))\neq 0$. 
Choose separating curves $\alpha_{1}$, $\alpha_{2}$, $\beta_{1}$ and $\beta_{2}$ in $S$ as described in Figure \ref{fig-curves30}.
It follows from $i(a, b)\neq 0$ that there exist $j, k\in \{ 1, 2\}$ with $i(a, \alpha_{j})\neq 0$ and $i(a, \beta_{k})\neq 0$.
Superinjectivity of $\phi$ implies $i(\phi(a), \phi(\alpha_j))\neq 0$ and $i(\phi(a), \phi(\beta_k))\neq 0$.
Since $\phi(\alpha_j)$ and $\phi(\beta_k)$ are curves in distinct components of $S_{\Phi(b)}$, we have $i(\Phi(a), \Phi(b))\neq 0$.

We next prove that $\Phi$ is injective on $V_{bp}(S)$, the set of vertices of $\calt(S)$ corresponding to BPs in $S$. 
Let $b$ and $c$ be BPs in $S$ with $\Phi(b)=\Phi(c)$. 
Lemma \ref{lem-30-Phi-link} shows that both of the maps
\begin{align*}
\Phi_b&\colon \lk_t(b)\cap \calc_s(S)\rightarrow \lk_t(\Phi(b))\cap \calc_s(S),\\
\Phi_c&\colon \lk_t(c)\cap \calc_s(S)\rightarrow \lk_t(\Phi(c))\cap \calc_s(S)
\end{align*}
defined as the restriction of $\Phi$ are surjective, where $\lk_t(d)$ denotes the link of $d$ in $\calt(S)$ for each BP $d$ in $S$. 
The images of $\Phi_b$ and $\Phi_c$ are then equal. 
Since the map $\phi \colon \calc_s(S)\rightarrow \calc_s(S)$ is injective, we obtain the equality $b=c$.

Note that for any $b, c\in V_{bp}(S)$, we have $b\neq c$ if and only if $i(b, c)\neq 0$.
Injectivity of $\Phi$ on $V_{bp}(S)$ implies $i(\Phi(b), \Phi(c))\neq 0$ for any $b, c\in V_{bp}(S)$ with $i(b, c)\neq 0$. 
The lemma then follows.     
\end{proof}

The last lemma and Proposition \ref{prop-30-t} imply the following:

\begin{thm}
Let $S=S_{3, 0}$ be a surface.
Then any superinjective map from $\calc_{s}(S)$ into itself is induced by an element of $\mod^*(S)$.
\end{thm}


\section{Proof of Theorem \ref{thm-t}}\label{sec-tor}

Let $S$ be the surface in Theorem \ref{thm-t}, and let $\phi \colon \calt(S)\rightarrow \calt(S)$ be a superinjective map.
We now prove that $\phi$ is induced by an element of $\mod^*(S)$.
It is shown in Lemma 3.7 and Proposition 3.16 of \cite{kida} that $\phi$ preserves vertices corresponding to separating curves and BPs in $S$, respectively. 
Applying Theorem \ref{thm-cs} (i) to the restriction of $\phi$ to $\calc_s(S)$, we can find $\gamma \in \mod^*(S)$ such that the equality $\phi(a)=\gamma a$ holds for any $a\in V_s(S)$.

We define a simplicial map $\phi_0\colon \calt(S)\rightarrow \calt(S)$ by setting $\phi_0(a)=\gamma^{-1}\phi(a)$ for each vertex $a$ of $\calt(S)$.
For each BP $b$ in $S$, one can find a collection $F$ of finitely many separating curves in $S$ such that $b$ is the only BP in $S$ disjoint from any curve in $F$ (see Figure \ref{fig-curves30} for example).
Since $\phi_0$ is the identity on $F$, it also fixes $b$.
It follows that $\phi_0$ is the identity and that $\phi$ is induced by $\gamma$.

We have proved assertion (i) of Theorem \ref{thm-t}.
We omit the proof of assertion (ii) of Theorem \ref{thm-t} because assertion (ii) can be derived from assertion (i) along the argument in Section 5 of \cite{bm} and Section 6.3 of \cite{kida}.



\begin{thebibliography}{99}

\bibitem{be-m}J. Behrstock and D. Margalit, Curve complexes and finite index subgroups of mapping class groups, {\it Geom. Dedicata} {\bf 118} (2006), 71--85.

\bibitem{bm-ar}R. W. Bell and D. Margalit, Injections of Artin groups, {\it Comment. Math. Helv.} {\bf 82} (2007), 725--751.

\bibitem{bm}T. Brendle and D. Margalit, Commensurations of the Johnson kernel, {\it Geom. Topol.} {\bf 8} (2004), 1361--1384.

\bibitem{bm-add}T. Brendle and D. Margalit, Addendum to: Commensurations of the Johnson kernel, {\it Geom. Topol.} {\bf 12} (2008), 97--101.

\bibitem{farb-ivanov}B. Farb and N. V. Ivanov, The Torelli geometry and its applications: research announcement, {\it Math. Res. Lett.} {\bf 12} (2005), 293--301.

\bibitem{flp}A. Fathi, F. Laudenbach and V. Po\'enaru, {\it Travaux de Thurston sur les surfaces.} S\'eminaire Orsay, Ast\'erisque, 66--67. Soc. Math. France, Paris, 1979.

\bibitem{gervais}S. Gervais, A finite presentation of the mapping class group of a punctured surface, {\it Topology} {\bf 40} (2001), 703--725. 

\bibitem{harvey}W. J. Harvey, Boundary structure of the modular group, in {\it Riemann surfaces and related topics: Proceedings of the 1978 Stony Brook Conference}, 245--251, Ann. of Math. Stud., 97, Princeton Univ.\ Press, Princeton, N.J., 1981.

\bibitem{irmak1}E. Irmak, Superinjective simplicial maps of complexes of curves and injective homomorphisms of subgroups of mapping class groups, {\it Topology} {\bf 43} (2004), 513--541.

\bibitem{irmak2}E. Irmak, Superinjective simplicial maps of complexes of curves and injective homomorphisms of subgroups of mapping class groups II, {\it Topology Appl.} {\bf 153} (2006), 1309--1340.

\bibitem{irmak-ns}E. Irmak, Complexes of nonseparating curves and mapping class groups, {\it Michigan Math. J.} {\bf 54} (2006), 81--110.

\bibitem{iva-aut}N. V. Ivanov, Automorphisms of complexes of curves and of Teichm\"uller spaces, {\it Int. Math. Res. Not.} {\bf 1997}, no. 14, 651--666. 

\bibitem{johnson}D. Johnson, Homeomorphisms of a surface which act trivially on homology, {\it Proc. Amer. Math. Soc.} {\bf 75} (1979), 119--125.

\bibitem{kls}R. P. Kent IV, C. J. Leininger and S. Schleimer, Trees and mapping class groups, {\it J. Reine Angew. Math.} {\bf 637} (2009), 1--21.

\bibitem{kida}Y. Kida, Automorphisms of the Torelli complex and the complex of separating curves, {\it J. Math. Soc. Japan} {\bf 63} (2011), 363--417.

\bibitem{kork-aut}M. Korkmaz, Automorphisms of complexes of curves on punctured spheres and on punctured tori, {\it Topology Appl.} {\bf 95} (1999), 85--111.

\bibitem{luo}F. Luo, Automorphisms of the complex of curves, {\it Topology} {\bf 39} (2000), 283--298.

\bibitem{mv}J. D. McCarthy and W. R. Vautaw, Automorphisms of Torelli groups, preprint,\\ arXiv:math/0311250.

\bibitem{powell}J. Powell, Two theorems on the mapping class group of surfaces, {\it Proc. Amer. Math. Soc.} {\bf 68} (1978), 347--350.

\bibitem{putman-conn}A. Putman, A note on the connectivity of certain complexes associated to surfaces, {\it Enseign. Math. (2)} {\bf 54} (2008), 287--301. 

\bibitem{sha}K. J. Shackleton, Combinatorial rigidity in curve complexes and mapping class groups, {\it Pacific J. Math.} {\bf 230} (2007), 217--232.


\end{thebibliography}
\end{document}